\documentclass{article}

\usepackage{import}
\usepackage{paperKM}

\graphicspath{{images/}{../images/}}

\title{ Post-Processed Posteriors for Sparse Covariances and Its Application to Global Minimum Variance Portfolio}
\author[1]{Kwangmin Lee}
\author[1]{Jaeyong Lee}
\affil[1]{Department of Statistics, Seoul National University}

\begin{document}
	
	\maketitle

	\begin{abstract}
		We consider Bayesian inference of sparse covariance matrices and propose a post-processed posterior. This method consists of two steps. In the first step, posterior samples are obtained from the conjugate inverse-Wishart posterior without considering the sparse structural assumption. The posterior samples are transformed in the second step to satisfy the sparse structural assumption through the hard-thresholding function. This non-traditional Bayesian procedure is justified by showing that the post-processed posterior attains the optimal minimax rates. We also investigate the application of the post-processed posterior to the estimation of the global minimum variance portfolio. We show that the post-processed posterior for the global minimum variance portfolio also attains the optimal minimax rate under the sparse covariance assumption. 
		The advantages of the post-processed posterior for the global minimum variance portfolio are demonstrated by a simulation study and a real data analysis with S\&P 400 data.

	\end{abstract}

	\section{Introduction}
	
	Suppose $X_1,\ldots,X_n$ are independent and identically distributed from a $p$-dimensional multivariate normal distribution with mean zero and a covariance matrix. We consider the estimation of the covariance matrix in this paper. When $p$ is larger than $n$, we refer to this situation as the high-dimensional settings, in which traditional covariance estimation methods such as the sample covariance estimator and the Bayesian inference by the inverse-Wishart prior, are not consistent. 
\cite{marvcenko1967distribution} showed that the eigenvalues derived from the sample covariance are not consistent if $p/n\lra c \in(0,1)$. 
For the Bayesian inference on the high-dimensional covariance, \cite{lee2018optimal} showed that an obviously inappropriate degenerate prior is a minimax optimal prior for the class of unconstrained covariances. 
Thus, structural assumptions on the covariance matrix are commonly employed for the consistent covariance estimation.


	Among the structural assumptions on the covariance matrix, we focus on the sparse covariance assumption, which imply that a large portion of elements of the covariance is (near) zero.
	It is a natural assumption when quite a few pairs of random variables in a random vector $X_i$ are perceived to be marginally independent.
	Sparse covariance estimation methods were introduced in both frequentist and Bayesian perspectives.
	\cite{bickel2008covariance} proposed the thresholded sample covariance estimator, which is constructed by hard-thresholding elements of the sample covariance matrix,
	and \cite{rothman2009generalized} generalized the thresholded sample covariance estimator so that other thresholding transformations such as soft-thresholding and SCAD are covered. 
	\cite{cai2012optimal} have shown the minimax optimality of the thresholded sample covariance estimator under a class of sparse covariances.
	A Bayesian method for sparse covariances is also proposed by \cite{lee2021beta}. They proposed the Beta-mixture shrinkage prior distribution for the covariance matrix and showed the convergence rate of the posterior distribution. According to this posterior convergence rate, this Bayesian method is justified only when $n$ is smaller than $p$. Thus, a Bayesian method having theoretical support under high-dimensional settings is required.

	We propose a post-processed posterior as a non-traditional Bayesian method for the class of sparse covariances.
	The procedure of post-processed posterior consists of two steps as follows. 
	First, the initial posterior sample is generated from the initial posterior distribution, the conjugate inverse-Wishart distribution. The sparse covariance assumption is not considered in this step. Second, the initial posterior samples are transformed via the hard-thresholding function used in \cite{cai2012optimal}. We call the post-processed posterior (PPP) {\it the thresholding post-processed posterior (thresholding PPP)} when we need to clarify the post-processing function. 
	We justify this procedure by showing that the post-processed posterior has the minimax optimal convergence rate under the sparse covariance assumption.
	
	The main feature of the post-processed posterior is that posterior samples are transformed to fit in the constrained parameter space.
	This idea has been employed in various constrained parameter spaces. 
	For example, \cite{dunson2003bayesian} and \cite{lin2014bayesian} used this idea for the ordered finite dimensional space and the monotone continuous function space, respectively.
	\cite{chakraborty2020convergence} also used this idea for the monotone measurable function space and showed the posterior consistency of their methods.
	This idea also was used for constrained covariance parameter spaces other than the sparse covariance structure in \cite{lee2020post} and \cite{lee2021condmean}. 
	The present work expands the application of this idea to the sparse covariance space with the decision-theoretic support. 
	

	We apply the post-processed posterior to the estimation of the global minimum variance portfolio (GMVP).
	Suppose that there are $p$ assets and the returns of $p$ assets are expressed as a $p$-dimensional random vector $X$ with covariance $\Sigma_0$. We consider a potfolio as an asset allocation vetor $w\in \bbR^p$ with $w^T\mathbf{1}=1$, where $ \mathbf{1}$ is the $p$-dimensional column vector of ones. 
	The global minimum variance portfolio is defined as the portfolio vector $w$ such that the expected risk of return is minimized.  
	Given a portfolio $w$, the return is $w^T X$ and the expected risk is defined as $Var(w^TX) = w^T \Sigma_0 w$. Thus, the global minimum variance portfolio is 
	$argmin_{w} w^T \Sigma_0 w$, which is given as 
	\bea
	w_{GMVP}(\Sigma_0) :=   \frac{\Sigma_0 \mathbf{1} }{\mathbf{1}^T\Sigma_0 \mathbf{1}},
	\eea
	\citep[See][]{merton1972analytic}.
	Since the GMVP is derived from $\Sigma_0$, 
	covariance estimation methods are commonly employed for the estimation of the global minimum variance portfolio.

	The traditonal estimator for the GMVP is the plug-in estimator with the sample covariance as 
	\bea
	w_{GMVP}(S_n) = \frac{S_n \mathbf{1} }{ \mathbf{1}^TS_n \mathbf{1} },
	\eea
	where $S_n = n^{-1} \sum_{i=1}^n X_i X_i^T $ is the sample covariance matrix.
	\cite{okhrin2006distributional} also derived the sampling distribution of this plug-in estimator. However, since the sample covariance matrix is singular when $p$ is larger than $n$, this estimator is not feasible for the high-dimensional settings. 
	To overcome the difficulty of high-dimensional settings, shrinkage type estimators were suggested. This type of estimators reduce the risk of estimators by allowing bias of estimators. For example, \cite{frahm2010dominating} and \cite{bodnar2018estimation} proposed the dominating estimator and Bona fide estimator, respectively, as the shrinkage type estimators. 
	
	There are also researches on the Bayesian inference on the global minimum variance portfolio. Once the posterior distribution on the covariance matrix is obtained, the posterior distribution on the GMVP is easily derived. 
	The Bayesian inference on the covariance matrix with inverse-Wishart prior is commonly used for the GMVP and has been largely investigated and used, for instance, in \cite{barry1974portfolio}, \cite{klein1977effect} and \cite{stambaugh1997analyzing}, to name a few. 
	Instead of assigning a prior distribution on the covariance matrix, \cite{bodnar2017bayesian} proposed a Bayesian method to assign prior distributions on the GMVP parameter $w_{GMVP}(\Sigma_0)$ directly. However, this method can not be used for the high-dimensional settings since the corresponding posterior distributions are not proper when $p$ is larger than $n$. 

	By applying the thresholding post-processed posterior to the GMVP, we derive the posterior distribution on the GMVP parameter. This posterior distribution has a minimax optimal convergence rate under the sparse covariance assumption. Since the result of minimax analysis is valid even in high-dimensional settings, the proposed method can be used when lots of assets are considered. 
	
	The main contributions of this paper are summarized as follows:
	\begin{itemize}
	\item We suggest a Bayesian method for high-dimensional sparse covariances with the optimal minimax rate. Note that the previous work \citep{lee2021beta} gives the convergence rate under the assumption $p<n$. 
	\item We suggest a Bayesian method for estimating the global minimum variance parameter under the high-dimensional sparse covariance assumption with the optimal minimax rate. This aspect of our work is new since there are no results on the minimax rate for the GMVP parameter under the high-dimensional settings. 
	\item Our method makes the interval estimation for the GMVP parameter under the high-dimensional sparse covariance assumption. Although \cite{bodnar2017bayesian} showed the performance of the interval estimation for the GMVP parameter by the simulation study, they considered the case when $p<n$ in the simulation study. 
	\end{itemize}


	The rest of this paper is structured as follows.
	In Section \ref{sec:thresPPP}, the algorithm of the thresholding post-processed posterior is given, and we show the minimax optimality of the method.
	In Section \ref{sec:portfolio}, the application of the thresholding post-processed posterior to the global minimum variance portfolio analysis is introduced with the result of minimax optimality.
	In Section \ref{sec:thresnumerical}, the thresholding post-processed posterior is demonstrated via a simulation study and data analysis of S\&P 400 data. The proofs of theorems are given in the supplementary material.

	\section{Thresholding Post-Processed Posterior}\label{sec:thresPPP}
	\sse{Notation}
	For any positive sequences $a_n$ and $b_n$, we denote $a_n = o(b_n)$ if $a_n / b_n \lra 0$ as $n\to\infty$, and $a_n \lesssim b_n$ if there exists a constant $C>0$ such that $a_n \le C b_n$ for all sufficiently large $n$. We denote $a_n \asymp b_n$ if $a_n \lesssim b_n$ and $b_n \lesssim a_n$. For a given matrix $A\in \bbR^{p\times p}$, $\lambda_{\min}(A)$ and $\lambda_{\max}(A)$ denote the minimum and maximum eigenvalues of $A$ and $[A]_i$ denote $i$th row of $A$. Let $||A|| = \{\lambda_{\max}(B^T B) \}^{1/2}$ be the spectral norm of $A \in \bbR^{p\times p}$ and $||A||_q$ be the matrix q-norm defined as 
	\bea
	\argmax_{x \neq 0}\frac{||Ax||_q}{||x||_q},
	\eea
	where $||x||_q$ represents the vector $q$-norm, i.e. $||x||_q = (\sum_{i=1}^p |x_i|^q )^{1/q}$ for $x=(x_1,x_2,\ldots,x_q)$. 
	 For any positive integer $p$, let $[p]=\{1,2,\ldots,p\}$.

	\sse{Algorithm of the post-processed posterior}
	Suppose $X_1,  \ldots, X_n$ are independent and identically distributed from the $p$-dimensional multivariate normal distribution with zero mean vector and covariance matrix $\Sigma_0\in \bbR^{p\times p}$, $N_p(0, \Sigma_0)$. We assume $\Sigma_0$ is an element of a class of sparse covariances $\mathcal{G}_q (c_{n,p},M_0,M_1)$ defined as
	$$\mathcal{G}_q (c_{n,p},M_0,M_1) = \{\Sigma=(\sigma_{ij})\in \calC_p : \sigma_{-j,j} \in B_q^{p-1} (c_{n,p}),\sigma_{jj}\le M_0,1\le j\le p , \lambda_{\min}(\Sigma)>M_1 \},$$ 
	where $M_0$, $M_1$ and $c_{n,p}$ are positive real numbers, $\calC_p$ is the set of all $p\times p$-dimensional positive definite matrices, and 
	$$B_q^{p-1} (c) = \{ \xi\in\mathbb{R}^{p-1} : |\xi|^q_{(k)}\le ck^{-1},\text{ for all } k=1,\ldots,p \}, $$
	where $|\xi|_{(k)}$ is the $k$th largest element in the absolute values of the elements of $\xi$. 
	This class of sparse covariances is the same as the class in \cite{cai2012optimal} except the minimum eigenvalue condition.
	
	We propose a post-processed posterior for the class of sparse covariances.
	
	\begin{enumerate}[(a)]
		\item (Initial posterior sampling step)
		We take the inverse-Wishart prior $ IW_p(B_0,\nu_0)$ as the initial prior, whose density function is 
		$$ \pi^i (\Sigma) \propto |\Sigma |^{-\nu_0/2} e^{-\half tr(\Sigma^{-1} B_0)}, \,\,  \Sigma \in \calC_p, $$
		where $B_0 \in \calC_p$ and $\nu_0 > 2p$. 
		The initial posterior distribution is then given as 
		$$\Sigma  \mid  \bbX_n  \sim IW_p(B_0 + n S_n,\nu_0 +n) ,  $$
		where $\bbX_n =(X_1,\ldots, X_n)^T$. 
		Initial posterior samples $\Sigma^{(1)},\ldots, \Sigma^{(N)}$ are generated from the initial posterior distribution.
		\item (Post-processing step)
		The initial posterior samples are transformed via the positive-definite adjusted thresholding function defined as 
		\bea
		H_\gamma^{(\epsilon_n)}(\Sigma) = \begin{cases}
			H_\gamma(\Sigma) +  \Big[ \epsilon_n - \lambda_{\min}\{H_\gamma(\Sigma)\} \Big] I_p &\quad\text{ if } \lambda_{\min}\{H_\gamma(\Sigma)\}<\epsilon_n, \\
			H_\gamma(\Sigma) &\quad\text{ otheriwse}.
		\end{cases},
		\eea
		for positive constants $\gamma$ and $\epsilon_n$, where $H_\gamma(\Sigma)$ is the element-wise hard thresholding function as 
		\bea
		H_\gamma(\Sigma) = \Big(\sigma_{ij} I\Big( |\sigma_{ij}| \ge \gamma \sqrt{\frac{\log p}{n}}\Big)\Big),
		\eea
		where $\sigma_{ij}$ is the $(i,j)$ element of $\Sigma$.
	\end{enumerate}
	
	\sse{Minimax analysis of the post-processed posterior}
	The post-processed posterior is justified in the decision-theoretic perspective. 
	We use the extended P-risk framework, which was introduced by \cite{lee2020post}. This framework is an extension of the P-risk framework in \cite{lee2018optimal} to incorporate post-processed posteriors. When there is no confusion, we refer to the extended P-risk framework as the P-risk framework and review the extended P-risk framework in this section.
	In the framework, we consider the parameter space and the data as $\calG_q$ and $\bbX_n$, respectively.  
	The action of decision theory is post-processed posterior $\pi^{i}(\cdot |\bbX_n; f)$ for a post-processing function $f$ and an initial prior $\pi^{i}$, and the decision rule is a pair of initial prior and post-processing function. 
	Note that a decision rule is defined as a mapping from the data space to the action space, and the post-processed posterior is obtained by combining a pair of initial prior and post-processing functions with data.
	The posterior-loss (P-loss) and posterior-risk (P-risk) are defined as 
	\bea
	\calL(\Sigma_0, \pi^{pp}(\cdot |\bbX_n; f) ) 
	& = & E^{\pi^i} ( ||\Sigma_0 - f(\Sigma) || \big| \bbX_n)  , \\
	\calR(\Sigma_0, \pi^{pp}) & = & E_{\Sigma_0} \calL(\Sigma_0, \pi^{pp}(\cdot |\bbX_n; f) )  \\
	& = & E_{\Sigma_0} E^{\pi^i} ( ||\Sigma_0 - f(\Sigma) || \big| \bbX_n),
	\eea
	where $E_{\Sigma_0}$ and $E^{\pi^i}(\cdot;\bbX_n)$ denote the expectations with respect to $\bbX_n$ and the initial posterior distribution, respectively.
	
	Given the definition of P-risk, we define the P-risk minimax rate. 
	For a real-valued sequence $r_n$, if 
	\bea
	\inf_{\hat{\Sigma}}\sup_{\Sigma_0\in \calG_q} E_{\Sigma_0} E^{\pi^i} ( ||\Sigma_0 - f(\Sigma) || \big| \bbX_n)  \asymp r_n,
	\eea
	as $n\lra \infty$,
	then $r_n$ is the minimax convergence rate. 
	A pair of prior and post-processing function $(\pi^*,f^*)$ is said to attain the P-risk minimax rate $r_n$, if 
	\bea
	\sup_{\Sigma_0\in \calG_q} E_{\Sigma_0} E^{\pi^*} ( ||\Sigma_0 - f^*(\Sigma) || \big| \bbX_n)  \asymp r_n,
	\eea
	as $n\lra \infty$.
	
	Using the extended P-risk framework, we show that the post-processed posterior has the optimal minimax rate. 
	First, we show the upper bound of the convergence rate of the method under the spectral norm using Theorem \ref{thm:upperbound}. If we set $\epsilon_n$ such that $\epsilon_n^2 	\lesssim c_{n,p} (\log p / n)^{(1-q)} +(\log p)/n $, then the convergence rate is $ c_{n,p} (\log p / n)^{(1-q)} +(\log p)/n$, which is the same as the lower bound of the frequentist minimax risk \citep{cai2012optimal}.
	Based on the second remark in \cite{lee2018optimal}, the minimax rate of P-risk is larger than or equal to the frequentist minimax rate, i.e. the lower bound of P-risk minimax rate is also $ c_{n,p} (\log p / n)^{(1-q)} +(\log p)/n$. Thus this rate turns out to be the P-risk minimax rate.
	Thus, the post-processed posterior has minimax optimal convergence rate. 
	\begin{theorem}\label{thm:upperbound}
		Let the prior $\pi^i$ of $\Sigma$ be $IW_p(A_n,\nu_n)$. If $(\nu_n-2p)\vee||A_n||_2\vee \log p=o(n)$ and $\gamma$ is a sufficiently large constant, then, there exists a positive constant $C$ such that 
		$$\bbE_{\Sigma_0}\bbE^{\pi^i} (||H_\gamma^{(\epsilon_n)}(\Sigma)-\Sigma_0||_2^2 |\bbX_n) \le C \Big(c_{n,p}^2 \Big( \frac{\log p}{n}\Big)^{(1-q)} +\frac{\log p}{n} + \epsilon_n^2\Big),$$
		for all sufficiently large $n$.
	\end{theorem}
	The proof of Theorem \ref{thm:upperbound} is given in the supplementary material.

	\section{Application to Bayesian Inference of the Global Minimum Variance Portfolio}\label{sec:portfolio}

	We apply the thresholding post-processed posterior to the Bayesian inference on the global minimum variance portfolio. 
	Suppose $p$-dimensional return vectors $X_1,\ldots,X_n$ are generated from 
	$N_p(0,\Sigma_0)$ and assume $\Sigma_0 \in \calG_q$. 
	Given $N$ samples of the post-processed posterior $H_\gamma^{(\epsilon_n)}(\Sigma^{(1)}),\ldots,H_\gamma^{(\epsilon_n)}(\Sigma^{(N)})$, we suggest the post-processed posterior on the GMVP parameter $w_{GMVP}(\Sigma_0)$ as 
	\bea
	w_{GMVP}(H_\gamma^{(\epsilon_n)}(\Sigma^{(1)})),\ldots,w_{GMVP}(H_\gamma^{(\epsilon_n)}(\Sigma^{(N)})).
	\eea

	We show that the post-processed posterior for the global minimum variance portfolio has the minimax optimal convergence rate. 
	For the minimax analysis, we define a P-risk on the $\bbR^p$, in which portfolio weights reside, as 
	\bea
	E_{\Sigma_0}\Big\{E^{\pi^{i}}\Big(\frac{||w_{GMVP}(\Sigma) - w_{GMVP}(\Sigma_0)||^2}{ ||w_{GMVP}(\Sigma_0)||^2}  \mid\bbX_n\Big)\Big\}.
	\eea
	We use the loss function $||w_{GMVP}(\Sigma) - w_{GMVP}(\Sigma_0)||^2/||w_{GMVP}(\Sigma_0)||^2$ instead of $||w_{GMVP}(\Sigma) - w_{GMVP}(\Sigma_0)||^2$, since arguments on the consistency based on the loss function $||w_{GMVP}(\Sigma) - w_{GMVP}(\Sigma_0)||^2$ are not rational when $p$ goes to infinity. 
	We briefly explain the problem here. We have 
	\bea
	||w_{GMVP}(\Sigma) - w_{GMVP}(\Sigma_0)||^2 &\le& 2 ||w_{GMVP}(\Sigma) ||^2  + 2 ||w_{GMVP}(\Sigma_0)||^2\\
	&\le&  2||w_{GMVP}(\Sigma) ||_1  ||w_{GMVP}(\Sigma) ||_\infty + 2 ||w_{GMVP}(\Sigma_0)||^2 \\
	&\le & 2||w_{GMVP}(\Sigma) ||_\infty + \frac{2 ||\Sigma_0||^2~ || \Sigma_0^{-1}||^2}{p}
	\eea
	This inequality shows that if $||w_{GMVP}(\Sigma) ||_\infty$ goes to zero, then $w_{GMVP}(\Sigma)$ is a consistent estimator. However, even the trivial estimator, equal weight portfolio $(1/p,\ldots,1/p)$, satisfies the condition. Hence, this loss function makes most estimators consistent, which does not seem to be.

	We show the convergence rate of the post-processed posterior for the global minimum variance portfolio and its minimax optimality with the P-risk.
	In Theorem \ref{thm:portfolio}, we show the upper bound of the convergence rate.
	\begin{theorem}\label{thm:portfolio}
		Let the prior $\pi^i$ of $\Sigma$ be $IW_p(A_n,\nu_n)$. If $(\nu_n-2p)\vee||A_n||_2\vee \log p=o(n)$, $\epsilon_n < \lambda_{\min}(\Sigma_0)$,   $c_{n,p}((\log p)/n)^{(1-q)/2} + 4\gamma ((\log p)/n)^{1/2} \lra 0$ as $n\lra \infty$ and $\gamma$ is a sufficiently large constant, then there exists a positive constant $C$ such that    
		\bea
		E_{\Sigma_0}\Big\{E^{\pi^{i}}\Big(\frac{||w_{GMV}(\Sigma) - w_{GMV}(\Sigma_0)||^2 }{||w_{GMV}(\Sigma_0)||^2 } \Big| \bbX_n\Big)\Big\}\le 
		C\Big( \frac{1}{n} + c_{n,p}^2 \Big(\frac{\log p}{n} \Big)^{1-q}  \Big)
		\eea
		for all sufficiently large $n$.
	\end{theorem}
	The proof of Theorem \ref{thm:portfolio} is given in the supplementary material.

	Next, we give the minimax lower bound of the P-risk using Assouad's lemma \citep[Lemma 2 in ][]{cai2012optimal} and an extended Assouad's lemma introduced in Lemma 3 of \cite{cai2012optimal}.
	First, we apply Assouad's lemma to space $\calG^{(1)}$ defined as 
	\bean\label{calG1}
	\calG^{(1)} = \{\Sigma(\theta):  \Sigma(\theta)^{-1} = M^{-1} I_p + \sum_{m=1}^{p'} \theta_m \frac{\tau}{\sqrt{n}} I(i=j=m), \theta= (\theta_1,\theta_2,\ldots,\theta_{p'}) \in \{0,1\}^{p'} \},
	\eean 
	where $p'=\lfloor p/2 \rfloor$ and $M= M_0+M_1$. Since $\calG^{(1)}\subset \calG_q(c_{n,p},M_0,M_1)$ for all sufficiently large $n$ when $\tau$ is a constant,
	Assouad's lemma gives
	\bea
	&& \sup_{\Sigma_0\in \calG_q} 2^2 E_{\Sigma_{0}} ||w_{GMV}(\Sigma(\theta'))-w_{GMV}(\Sigma(\theta))||^2 \\
	&\ge& 
	\max_{\theta \in \{0,1\}^{p'} } 2^2E_{\Sigma(\theta)} ||w_{GMV}(\Sigma(\theta'))-w_{GMV}(\Sigma(\theta))||^2\\
	&\ge& \min_{H(\theta,\theta')\ge 1} \frac{||w_{GMV}(\Sigma(\theta'))-w_{GMV}(\Sigma(\theta))||^2}{H(\theta,\theta')} \frac{p'}{2} \min_{H(\theta,\theta')=1} ||\bbP_\theta \wedge \bbP_{\theta'}||.
	\eea 
	Using Lemma \ref{lemma:lowerbound1}, we obtain the first minimax lower bound as 
	\bean\label{formula:minimax1}
	\sup_{\Sigma_0\in \calG_q} 2^2 E_{\Sigma_{0}} ||w_{GMV}(\Sigma(\theta'))-w_{GMV}(\Sigma(\theta))||^2 &\ge& \frac{C}{np},
	\eean
	for a positive constant $C$.
	
	\begin{lemma}\label{lemma:lowerbound1}
		If $\tau/\sqrt{n} \le M/3$ and $\tau/M \le 1/3$, then there exists a positive constant $C$ such that
		\bea
		\min_{H(\theta,\theta')\ge 1} \frac{||w_{GMV}(\Sigma(\theta'))-w_{GMV}(\Sigma(\theta))||^2}{H(\theta,\theta')} \frac{p'}{2} \min_{H(\theta,\theta')=1} ||\bbP_\theta \wedge \bbP_{\theta'}|| \ge \frac{C}{np},
		\eea
		for all sufficiently large $n$. 
	\end{lemma}
	The proof of Lemma \ref{lemma:lowerbound1} is given in the supplementary material.
	
	Next, we apply the extended Assouad's lemma \citep[Lemma 3]{cai2012optimal}.
	Before using this lemma for the GMVP problem, we review this lemma here.
	Let $B\subset \bbR^p \setminus\{ \mathbf{0}\}$ be a finite set and let $\Lambda \subset B^r$ with a positive integer $r$. Then, we define a finite set $\Theta$ as
	\bea
	\Theta = \{0,1\}^r \otimes \Lambda.
	\eea
	We define projections $\gamma: \Theta \mapsto \{0,1\}^r$ and $\lambda: \Theta\mapsto \Lambda$ as
	$\gamma(\theta_0) = \gamma_0$ and $\lambda(\theta_0) = \lambda_0$, respectively, for
	$\theta_0 = (\gamma_0,\lambda_0)$, $\gamma_0 \in \{0,1\}^r$ and $\lambda_0\in \Lambda$. We also let $\gamma_i(\theta_0)$ and $\lambda_i(\theta_0)$ be
	$i$th element of $\gamma_0$ and  $i$th row-vector of $\lambda_0$, respectively.
	The extended Assouad's lemma (Lemma 3 in \cite{cai2012optimal}) gives
	\bea
	\max_{\theta\in\Theta} 2^2 E_{\Sigma(\theta)} d^2(T,\psi(\Sigma(\theta))) \ge 
	\min_{H(\gamma(\theta),\gamma(\theta'))\ge 1} \frac{d^2(\psi(\Sigma(\theta)),\psi(\Sigma(\theta')))}{H(\gamma(\theta),\gamma(\theta'))} \frac{r}{2}
	\min_{1\le i\le r} || \bar{\mathbb{P}}_{i,0} \wedge \bar{\mathbb{P}}_{i,1} ||,
	\eea
	where 
	\bea
	\bar{\mathbb{P}}_{i,j} &=& \frac{1}{2^{r-1} Card(\Lambda)} \sum_{\theta\in\Theta_{i,j}}\bbP_\theta,\\
	\Theta_{i,j} &=& \{\theta\in\Theta : \gamma_i(\theta) = j\},
	\eea
	for $j=0,1$ and $i=1,2,\ldots,r$.

	For the application of the extended Assouad's lemma, the elements $r$, $\Lambda$ and $\psi(\Sigma(\theta))$ need to be specified. 
	We set $r=\lfloor p/4 \rfloor$ and $\psi(\Sigma(\theta)) = w_{GMVP}(\Sigma(\theta))$.
	Given a $p$-dimensional vector $\lambda_i$, let $A_i(\lambda_i)$ be a $p\times p$ symmetric matrix with $i$th row and column are equal to $\lambda_i$ and zero for the other elements. We set $\Lambda$ as 
	\bea
	\Lambda = \Lambda_k := \{(\lambda_1,\lambda_2,\ldots,\lambda_r) \in B^r &:& (\lambda_i)_j=0 \text{ or }1, ||\lambda_i||_0 = k, ||(\lambda_i)_{1:(p-r)}||_1 =0,\\
	&& ||[\sum_{i=1}^r A_i(\lambda_i)]_{j}||_1 \vee ||[\sum_{i=1}^r A_i(\lambda_i)^T]_{j}||_1 \le 2k  , \text{ for } i\in[r], j\in[p] \}.
	\eea
	Given the notations, we define a covariance space $\calG^{(2)}$ as
	\bean\label{calG2}
	\calG^{(2)}(k,\epsilon_{n,p})  = \{ \Sigma(\theta) : \Sigma(\theta) = MI_p +  \epsilon_{n,p}\sum_{i=1}^r \gamma_i(\theta) A_i(\lambda_i(\theta)), \theta = (\gamma,\lambda)\in \{0,1\}^r \otimes \Lambda_k \}.
	\eean
	When $k= \max(\lfloor  c_{n,p}\epsilon_{n,p}^{-q}\rfloor,0)$, $c_{n,p} \le C_0 n^{(1-q)/2} (\log p)^{-(3-q)/2}$ for a positive constant $C_0$ and $\epsilon_{n,p} = (0.25(\log 2) \min(1,M_1)C_0^{-1})^{1/(1-q)} ((\log p)/n)^{1/2}$, we have
	$$2k\epsilon_{n,p} \le 0.5(\log 2) \min(1,M_1),$$
	 thus, $\calG^{(2)}(k,\epsilon_{n,p}) \subset \calG_q$.
	The extended Assouad's lemma gives
	\bea
	&& \sup_{\Sigma_0\in \calG_q} 2^2 E_{\Sigma_{0}} ||w_{GMV}(\Sigma(\theta'))-w_{GMV}(\Sigma(\theta))||^2  \\
	&\ge&
	\min_{H(\gamma(\theta),\gamma(\theta'))\ge 1} \frac{||w_{GMV}(\Sigma(\theta'))-w_{GMV}(\Sigma(\theta))||^2 }{H(\gamma(\theta),\gamma(\theta'))} \frac{r}{2}
	\min_{1\le i\le r} || \bar{\mathbb{P}}_{i,0} \wedge \bar{\mathbb{P}}_{i,1} ||.
	\eea
	Using Lemma \ref{lemma:lowerbound2}, we obtain 
	\bean
	\sup_{\Sigma_0\in \calG_q} 2^2 E_{\Sigma_{0}} ||w_{GMV}(\Sigma(\theta'))-w_{GMV}(\Sigma(\theta))||^2   \ge \frac{ C c_{n,p}^2}{p}  \Big(\frac{\log p}{n}\Big)^{1-q} ,\label{formula:minimax2}
	\eean
	for a positive constant $C$.
	\begin{lemma}\label{lemma:lowerbound2}
		Suppose the notation of $\calG^{(2)}$ in \eqref{calG2}. Assume $c_{n,p} \le C_0  n^{(1-q)/2} (\log p)^{-(3-q)/2}$ for some positive constant $C_0$. There exists a positive constant $C$ such that 
		\bea
		\min_{H(\gamma(\theta),\gamma(\theta'))\ge 1} \frac{||w_{GMV}(\Sigma(\theta))-w_{GMV}(\Sigma(\theta'))||^2}{H(\gamma(\theta),\gamma(\theta'))} \frac{r}{2}
		\min_{1\le i\le r} || \bar{\mathbb{P}}_{i,0} \wedge \bar{\mathbb{P}}_{i,1} || \ge \frac{ C c_{n,p}^2}{p}  \Big(\frac{\log p}{n}\Big)^{1-q} ,
		\eea
		for all sufficiently large $n$.
	\end{lemma}
	The proof of Lemma \ref{lemma:lowerbound2} is given in the supplementary material.
	
	Using Lemmas \ref{lemma:lowerbound1} and \ref{lemma:lowerbound2}, we obtain Theorem \ref{thm:minimaxlower}, the theorem of the minimax lower bound. The minimax lower bound is the same as the upper bound given in Theorem \ref{thm:portfolio}; thus, the post-processed posterior is minimax optimal in respect of the global minimum variance portfolio. 
	\begin{theorem}\label{thm:minimaxlower}
		Assume $c_{n,p} \le C_0  n^{(1-q)/2} (\log p)^{(3-q)/2}$ for some positive constant $C_0$. There exists a positive constant $C$ such that 
		\bea
		\inf_{\hat{\Sigma}}\sup_{\Sigma_0\in \calG_q} E_{\Sigma_{0}} \frac{||w_{GMV}(\hat\Sigma)-w_{GMV}(\Sigma_0)||^2 }{||w_{GMV}(\Sigma_0)||^2 }  \ge    C\Big( \frac{1}{n} + c_{n,p}^2 \Big(\frac{\log p}{n} \Big)^{1-q}  \Big),
		\eea 
		for all sufficiently large $n$.
	\end{theorem}
	\begin{proof}
		Since
		\bea
		||w_{GMV}(\Sigma_0)||^2  \le \frac{||\Sigma_0^{-1}|| ~||\Sigma_0||^2}{p},
		\eea
		the proof is completed by collecting the first and second lower bounds in \eqref{formula:minimax1} and \eqref{formula:minimax2}.
	\end{proof}
	
	In the comparison of the minimax rate with that under the spectral norm (Theorem \ref{thm:upperbound}),
	while the convergence rate under the loss function $p ||w_{GMV}(\Sigma) - w_{GMV}(\Sigma_0)||^2$ is $c_{n,p}^2 (\log p)/n + 1/n$, the convergence rate under the spectral norm is $c_{n,p}^2 ((\log p)/n)^{1-q} + (\log p)/n$. The term $1/n$ is replaced with $(\log p)/n$ for the convergence rate of the global minimum variance portfolio.

	\section{Simulation Studies}\label{sec:thresnumerical}
	
	\sse{Simulation Study}
	
	In this section, we compare the post-processed posterior with the thresholded sample covariance \citep{cai2012optimal}, the Bona fide estimator \citep{bodnar2018estimation}, and the conventional Bayesian method by the inverse-Wishart prior and Beta mixture shrinkage prior \citep{lee2021beta}.
Let $p=100$ and we consider two $100\times 100$ true covariances $\Sigma_0^{(1)}$ and $\Sigma_0^{(2)}$.
	We define $\Sigma_0^{(1)}=(\sigma_{0,ij}^{(1)}) + 0.1 I_p$ with 
	\begin{equation*}
		\sigma_{0,ij}^{(1)} = \begin{cases}
			0.1,               &  10k+1 \le i,j\le 10k+10, k \text{ is even numbers} \\
			4 , &  10k+1 \le i,j\le 10k+10, k \text{ is odd numbers} \\
			0 , &\text{otherwise},
		\end{cases}
	\end{equation*} 
	and $\Sigma_0^{(2)}=(\sigma_{0,ij}^{(2)}) + 0.1 I_p$ with 
	\begin{equation*}
		\sigma_{0,ij}^{(2)} = \begin{cases}
			0.25,               &  10k+1 \le i,j\le 10k+10, k=0,5 \\
			0.5,               &  10k+1 \le i,j\le 10k+10, k=1,6 \\
			1,               &  10k+1 \le i,j\le 10k+10, k=2,7 \\
			2,               &  10k+1 \le i,j\le 10k+10, k=3,8 \\
			4,               &  10k+1 \le i,j\le 10k+10, k=4,9 \\						
			0 , &\text{otherwise}.
		\end{cases}
	\end{equation*} 
	For each true covariance, the true global minimum variance portfolio is derived, and true covariances and the corresponding global minimum variance portfolios are visualized in Figure \ref{fig:Sigmas}. 
	\begin{figure}[!tb]
		\centering
		\includegraphics[height=8cm,width=16cm]{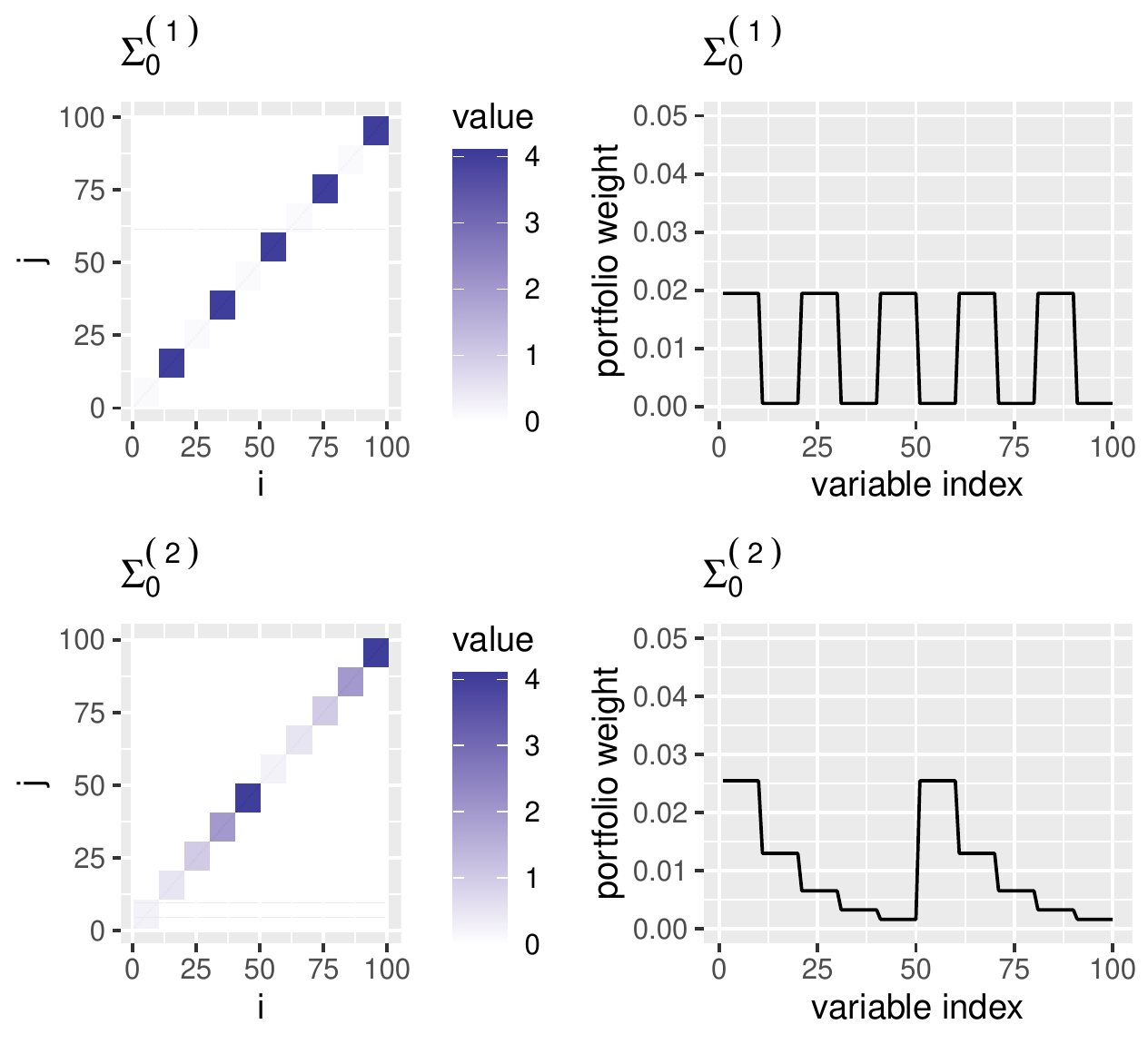}
		\caption{Visualization of covariance matrices and the derived global minimum variance portfolios of $\Sigma_0^{(1)}$ and $\Sigma_0^{(2)}$. The covariance matrices are in the left column, and vectors of global minimum variance portfolio weights are in the right column. The x-axis on the right plot represents the index of variables (assets), and the y-axis is the weight.}
		\label{fig:Sigmas}
	\end{figure}
	We generate the data $X_1,\ldots, X_n \iid N_p(0, \Sigma_0^{(t)})$ for $t=1$ and $2$, and $n=50,500$ and $2000$.
	For the Bayesian methods, we need to specify the hyperparameters and the number of posterior samples. 
	When we use the inverse-Wishart prior for the inverse-Wishart method and the initial prior of PPP, we set the shape parameter $\nu_0 = 2p+2$, the scale matrix as $\bar{s}_{ii}I_p$ where $\bar{s_{ii}} = p^{-1}\sum_{i=1}^p s_{ii}$ and $s_{ii}$ is the $i$th diagonal element of the sample covariance. We generate $2000$ posterior sample for the inverse-Wishart posteriors. For the beta-mixture shrinkage prior, we set the hyperparameter as suggested in \cite{lee2021beta} and generate $3000$ posterior sample including $1000$ burn-in sample.

	For the thresholded sample covariance and the post-processed posterior, tuning parameters of threshold parameter and positive-definite adjustment parameter need to be specified. We use the cross-validation idea for this purpose. 
We split the data into train and validation data and measure the prediction error using the validation data and the estimator based on the train data. 
Let $X^{train}$ and $X^{val}$ denote train data and validation data, respectively, and let 
$\hat{\Sigma}(X^{train})$ be the covariance estimator based on the training data.
We measure the prediction error as $f(\hat{\Sigma}(X^{train}),X^{val})$ for a loss function $f$, which is chosen in accordance with the goal of data analysis. We present two examples of this loss function in this section. 
If $\hat{\Sigma}$ is given as a set of posterior sample, $\Sigma^{(1)}, \ldots, \Sigma^{(N)}$, then we measure the prediction error as the posterior mean of prediction errors $N^{-1}\sum_{i=1}^N f(\Sigma^{(i)}, X^{val})$.

	
	
	We conduct the comparison study in the respect of estimating covariance itself. 
	For the tuning parameter selection, we set the prediction error function $f(\hat{\Sigma}(X^{train}),X^{val})$ as 
	\bea
	||\hat{\Sigma}(X^{train}) - (X^{val})^T X^{val}/ N_{val}||_2,
	\eea
	where $N_{val}$ is the number of observations in $X^{val}$. This error function was used for the covariance estimation in \cite{bickel2008regularized}. 
	For $50$ sets of simulated data, we calculate the error of each covariance estimators as
	\bean
	\frac{1}{50} \sum_{s=1}^{50}  ||\Sigma_0-\hat{\Sigma}^{(s)}||_2/||\Sigma_0||_2 ,
	\eean
	where $\hat{\Sigma}^{(s)}$ is an point estimator based on the $s$-th simulated data set. For Bayesian methods, we use posterior mean as the point estimator. Table \ref{tbl:point} gives the simulation errors.
	\begin{table}[h]
		\centering
		\begin{tabular}{|c|c|c|c|c|c|c|}
			\hline
			& \multicolumn{3}{c|}{$\Sigma_0^{(1)}$} & \multicolumn{3}{c|}{$\Sigma_0^{(2)}$} \\ \hline
			& n = 50     & n = 500    & n = 2000    & n = 50     & n = 500    & n = 2000    \\ \hline
			PPP       & 0.40   &    0.11 &  0.05  &      0.31      &    0.11        &       0.07      \\ \hline
			CGM     & 0.85   &    0.77    &  0.73  &     0.82       &   0.75         &        0.69     \\ \hline
			IW        & 0.55  & 0.17   &    0.06    &     0.40       &       0.13      &      0.07\\ \hline
			Thres     & 0.39  & 0.12   &    0.07     &     0.30       &    0.10        &    0.07 \\ \hline
			Sample cov & 0.55   & 0.17    &  0.09   &  0.40      &       0.13     &       0.07      \\ \hline
		\end{tabular}
		\caption{Errors of point estimators for the covariance under the spectral norm. PPP represents the our method. CGM and IW represent the Bayesian methods with the Beta-mixture shrinkage prior and the inverse-Wishart prior, respectively. Thres and Sample cov represent the frequentist methods with the thresholded sample covariance and the sample covariance, respectively.}
		\label{tbl:point}
	\end{table}
	The thresholding post-processed posterior and the thresholded sample covariance have the smallest errors in all settings.

	Next, we conduct the comparison study in terms of estimation of the global minimum variance portfolio. 
	For the tuning parameter selection, we set the prediction error function $f(\hat{\Sigma}(X^{train}),X^{val})$ as 
	\bea
	\hat{\sigma}(w_{GMVP}(\hat{\Sigma}),X^{val}) := \hat{Var}\{ w_{GMVP}(\hat{\Sigma})^T [X^{(val)}]_1,\ldots, w_{GMVP}(\hat{\Sigma})^T [X^{(val)}]_{Nval}\},
	\eea
	where $\hat{Var} A$ is the sample variance of set $A\subset\bbR$, and $[X^{(val)}]_i$ represents $i$th observation of $X^{(val)}$. 
	This loss function is constructed based on the fact that the true GMVP is $argmin_w Var(w^T X)$.
	We obtain the GMVP estimators using the simulated data and visualize the result of estimation of the Bayesian methods in Figures \ref{fig:GMV_Sigma1} and \ref{fig:GMV_Sigma2}. In these figures, posterior means and the $95\%$ credible intervals are represented. 
	\begin{figure}
		\centering
		\includegraphics[scale=1]{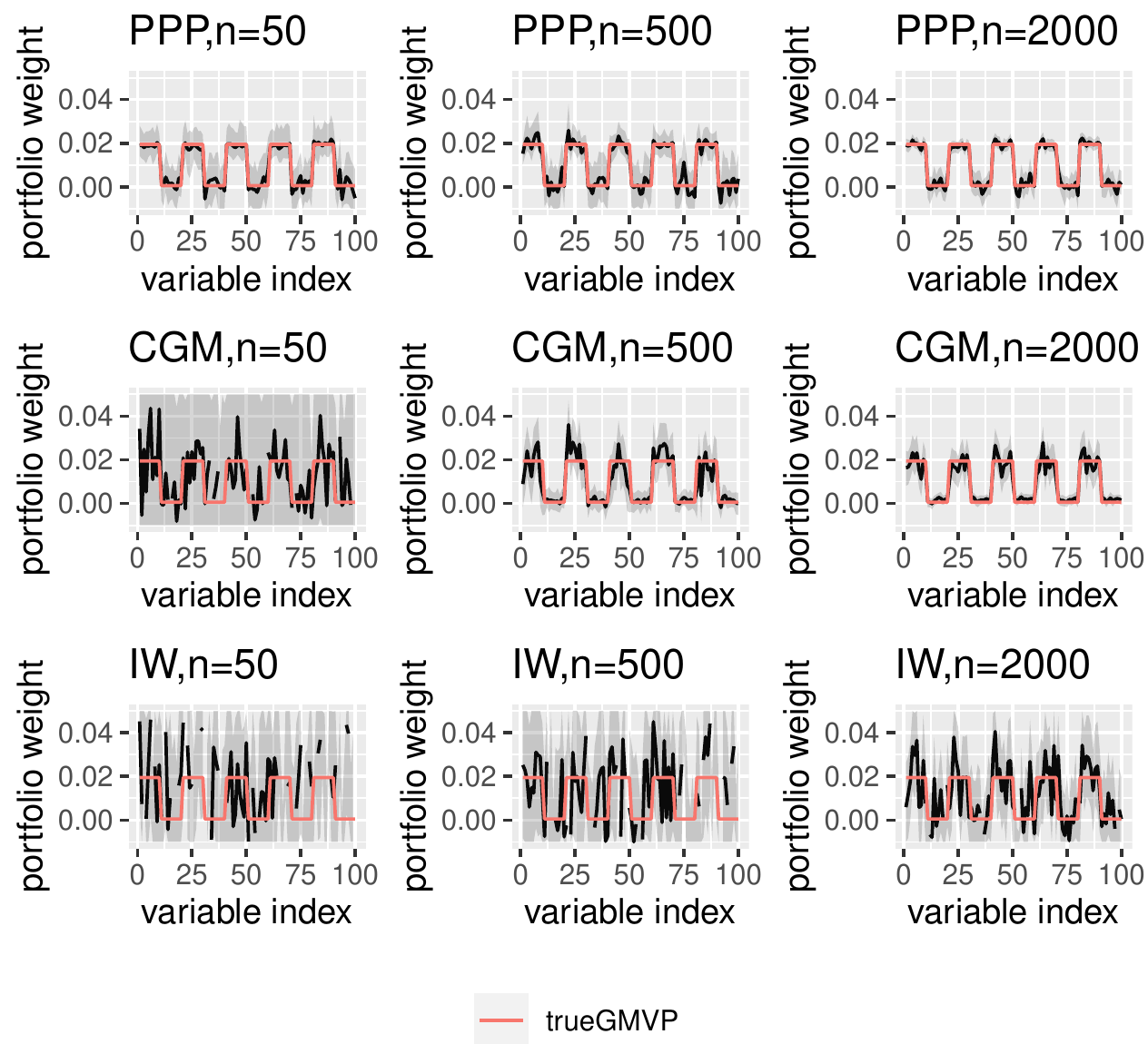}
		\caption{When the true covariance is $\Sigma_0^{(1)}$, the posterior mean and the $95\%$ credible intervals of elements in the global minimum variance portfolio are represented for the Bayesian methods: the post-processed posterior, conventional Bayesian methods with the Beta-mixture shrinkage prior and the inverse-Wishart prior.}
		\label{fig:GMV_Sigma1}
	\end{figure}
	\begin{figure}
		\centering
		\includegraphics[scale=	1]{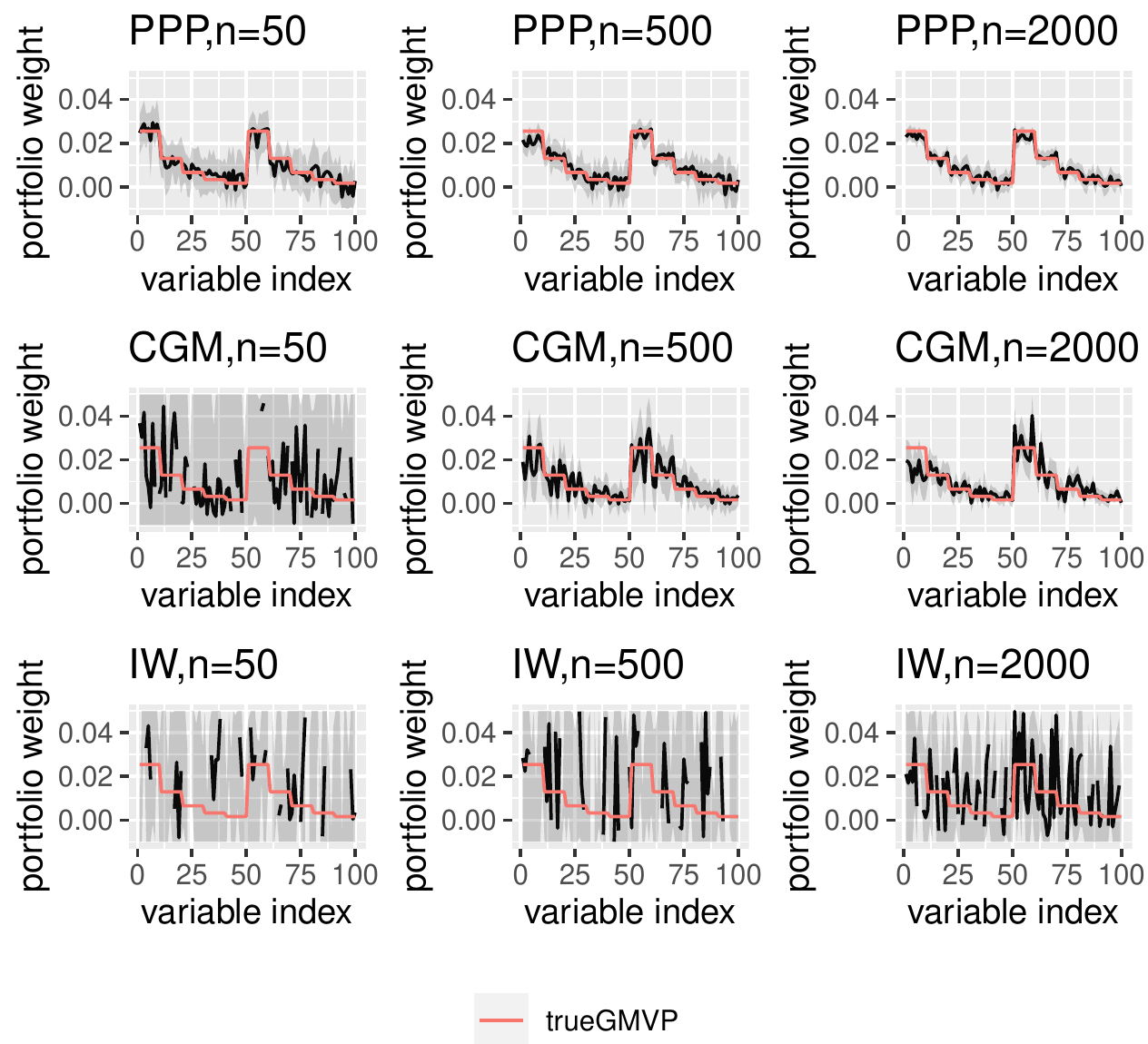}
		\caption{When the true covariance is $\Sigma_0^{(2)}$, the posterior mean and the $95\%$ credible intervals of elements in the global minimum variance portfolio are represented for the Bayesian methods: the post-processed posterior, conventional Bayesian methods with the Beta-mixture shrinkage prior, and the inverse-Wishart prior.}
		\label{fig:GMV_Sigma2}
	\end{figure}

	We repeat generating the simulation data $50$ times and give the summarized performance in both point and interval estimation aspects.
	We measure the point estimation error of the GMVP as
	\bea 
	\frac{1}{50} \sum_{s=1}^{50}  ||w_{GMVP}(\Sigma_0)-\hat{w}_{GMVP}^{(s)}||_2/||w_{GMVP}(\Sigma_0)||_2,
	\eea
	where $\hat{w}_{GMV}^{(s)}$ is the point estimator for the GMVP in the $s$-th simulated data set.
	The performances are summarised in Table \ref{tbl:gmv_point}. Table \ref{tbl:gmv_point} shows that the thresholding post-processed method has the smallest errors in all settings.
	\begin{table}[h]
		\centering
		\begin{tabular}{|c|c|c|c|c|c|c|}
			\hline
			& \multicolumn{3}{c|}{$\Sigma_0^{(1)}$} & \multicolumn{3}{c|}{$\Sigma_0^{(2)}$} \\ \hline
			& n = 50     & n = 500    & n = 2000    & n = 50     & n = 500    & n = 2000    \\ \hline
			PPP       & 0.22   &    0.15 &  0.09  &     0.29      &    0.24        &       0.14    \\ \hline
			CGM       & 1.06   &    0.28    &  0.22  &     2.14       &   0.38         &        0.40   \\ \hline
			IW        & 1.95  & 1.47   &    0.72    &     3.88     &       2.69     &      1.30\\ \hline
			Thres     & 0.30  & 0.18   &    0.12    &     0.38      &    0.27        &    0.16 \\ \hline
			Bona fide & 2.01   & 1.51    &  0.72   &  2.91     &       2.37    &       1.26     \\ \hline
		\end{tabular}
		\caption{Errors of point estimators for the global minimum variance portfolio.
		PPP represents the our method. CGM and IW represent the Bayesian methods with the Beta-mixture shrinkage prior and the inverse-Wishart prior, respectively. Thres and Bona fide represent the frequentist methods with the thresholded sample covariance and the Bona fide estimator, respectively.}
		\label{tbl:gmv_point}
	\end{table}
	
	We also show the performances in the respect of interval estimation for the GMVP. Given the posterior sample $w_{GMV}^{(1)},\ldots,w_{GMV}^{(N)}$, 
let $[l_i,u_i]$ denote $95\%$ credible interval of $i$th asset weight.
We measure the coverage probability as 
\bea
100\frac{1}{p}\sum_{i=1}^p I(w_{GMV}(\Sigma_0)_i \in [l_i,u_i]).
\eea
Table \ref{tbl:gmv_interval} gives the average of coverage probabilities.
The coverage probabilities by the thresholding post-processed posterior are nearer to the desired probability $95\%$ than those by the other methods.
	\begin{table}[]
		\centering
		\begin{tabular}{|c|c|c|c|c|c|c|}
			\hline
			& \multicolumn{3}{c|}{$\Sigma_0^{(1)}$} & \multicolumn{3}{c|}{$\Sigma_0^{(2)}$} \\ \hline
			& n = 50     & n = 500    & n = 2000    & n = 50     & n = 500    & n = 2000    \\ \hline
			PPP       & 95.8\%   & 96.8\%  &    96.9\%   &   97.0\%     &  95.3\%          &   93.9\%          \\ \hline
			CGM       &     100\%    &   97.0\%  &    94.3\%   &           99.9\%    &  95.4\%          &  80.9\%           \\ \hline
			IW        &       87.8\%     &  89.5\%   &  92.0\%   &        83.1\%    &  89.2\%          &   93.0\%          \\ \hline
		\end{tabular}
		\caption{Average coverage probabilities of the Bayesian methods. PPP represents the our method. CGM and IW represent the Bayesian methods with the Beta-mixture shrinkage prior and the inverse-Wishart prior, respectively. }
		\label{tbl:gmv_interval}
	\end{table}

	
	\sse{Real Data Analysis}
	
	We apply the GMVP estimators to S\&P 400 data, consisting of the stock price of $400$ mid-cap companies in the United States.
	We collect the data from \cite{ssga} and the collection period is from 2nd May 2011 to 30th April 2021.
	We process this data to obtain a monthly return and discard data of stocks that have missing values. As a result of preprocessing, we obtain $120\times 327$ data matrix $X$ in which the rows represent observations and the columns represent variables. 
	
	To compare the performances of the GMVP estimators, we extract the train and test data form the S\&P 400 data as follows: 
	\begin{enumerate}
		\item Sample an index $i$ from $\{ 48,49,\ldots,120-12  \}$
		\item Let $[X]_{(i-48+1):i}$ and $[X]_{(i+1):(i+12)}$ denote the train and test data, respectively,
		where $[X]_{j:k}$ is the submatrix of $X$ from $j$th row to $k$th row. 
	\end{enumerate}
	By setting the pair of train and test data, we use $4$ years data for the next $1$ year. 
	
	We repeat generating the pair of train and test data $20$ times and let $X^{(train),i}$ and $X^{(test),i}$ be the train and test data in the $i$th iteration. 
	Using these pairs of train and test data, we estimate the portfolio variance for the GMVP estimators as 
	\bea
	\frac{1}{20}\sum_{i=1}^{20} 100 \sqrt{\hat{Var} (\hat{w}(X^{(train),i}), X^{(test),i})},
	\eea
	which is represented in Table \ref{tbl:snp400}.
	\begin{table}[h]
		\centering
		\begin{tabular}{|c|c|c|c|c|}
			\hline
			& PPP  & Thres & Bona fide & IW \\ \hline
			Var & 3.50 & 3.49                   & 4.07      & 3.80            \\ \hline
		\end{tabular}
		\caption{ 
			The portfolio vector is estimated with the train data, and the portfolio variance is estimated with the test data and the estimated portfolio.
		The average value of estimated portfolio variance is represented. 
	PPP represents the our method. IW represent the Bayesian method with the inverse-Wishart prior. Thres and Bona fide represent the frequentist methods with the thresholded sample covariance and the Bona fide estimator, respectively.   }
		\label{tbl:snp400}
	\end{table}
	Table \ref{tbl:snp400} shows that the thresholded sample covariance and the post-processed posterior give the smallest variance. 
	

	\section{Discussion}

We have proposed a post-processed posterior for the Bayesian inference on sparse covariances and applied this method to estimate the global minimum variance portfolio.
The main advantage of the method over the frequentist methods is on interval estimations for functionals of covariance. The simulation study shows that the interval estimators on the global minimum variance parameter have the frequentist property, attaining nominal coverage probability on the true value. 
We have also shown that the proposed method has the minimax optimal convergence rates for the covariance and the global minimum variance portfolio parameter.

The present work extends the post-processing method to the class of sparse covariances and the minimax analysis on the global minimum variance portfolio parameter. Since the minimax analysis is limited to the sparsity assumption on the covariance, we expect to conduct the minimax analysis to other covariance assumptions as to future works.
For example, for the unconstrained covariance space, we can check the attainability of consistent estimation for the global minimum variance portfolio parameter as \cite{lee2018optimal} did for the covariance itself.  
	

	

	\bibliographystyle{dcu}
	\bibliography{cov-ppp}

\end{document}


\maketitle

In the supplementary material, we provide the proofs of lemmas and theorems in the main paper. 

We let $A_{\{i,j\}\times\{i,j\}}$ denote $2\times 2$ sub-matrix of $A$ with $i,j$ row and column for a matrix $A$. 
For a $p\times p$ symmetric matrix $A$, let $diag(A)$ be a $p\times p$ diagonal matrix in which diagonal elements are equal to those of $A$. We also let $diag^c(A) = A-diag(A)$.

\se{Proof of Theorem \ref{main-thm:upperbound} and Related Lemma}
In this section, we prove Theorem \ref{main-thm:upperbound}.
First, we present Lemma \ref{lem:ineq_thres} which gives an inequality for thresholded covariances. 

\begin{lemma}\label{lem:ineq_thres} 
Let $n,p$ and $r$ be positive integers and $\gamma$ be a positive real number. 
Suppose $\Sigma_0 = (\sigma_{0,ij})_{1\le i,j\le p} \in \calG_q(c_{n,p},M_0,M_1)$.
There exist a positive constant $C$ such that 
\bea
||H_\gamma(\Sigma)-\Sigma_0||_1^r &\le& 2^{r-1} ||D ||_1^r +  4^{r-1}\Big\{4C c_{n,p} \Big(\frac{\log p}{n}\Big)^{(1-q)/2}\Big\}^r + 
  \Big\{4\gamma \Big(\frac{\log p}{n}\Big)^{1/2}\Big\}^r,\\
 ||diag^c\{H_\gamma(\Sigma)-\Sigma_0\}||_1^r &\le& 2^{r-1} ||D ||_1^r +  2^{r-1}\Big\{4C c_{n,p} \Big(\frac{\log p}{n}\Big)^{(1-q)/2}\Big\}^r,
\eea
where $D=(d_{ij})_{1\le i,j\le p}$, $d_{ij}=(H_\gamma(\Sigma)_{ij}-\sigma_{0,ij}) I(A_{ij}^c)$, and 
\bea
A_{ij}=\Big\{ |(H_\gamma(\Sigma))_{ij}-\sigma_{0,ij}|\le 4\min \Big\{ |\sigma_{0,ij}|,\gamma\sqrt{\frac{\log p}{n}}\Big\} \Big\}.
\eea
		
\end{lemma}
\begin{proof}
We have 
\bea
||H_\gamma(\Sigma)-\Sigma_0||_1^r &\le& 
2^{r-1} ||D ||_1^r + 2^{r-1} || \{(H_\gamma(\Sigma)_{ij} - \sigma_{0,ij})I(A_{ij})   \}_{1\le i,j\le p} ||_1^r\\
&\le& 2^{r-1} ||D ||_1^r + 4^{r-1} [\sup_j \sum_{i\neq j} |(H_\gamma(\Sigma)_{ij}- \sigma_{0,ij}|I(A_{ij})]^{r} + 4^{r-1} \{\sup_j |(H_\gamma(\Sigma)_{jj} - \sigma_{0,jj}|I(A_{jj})\}^r\\
&\le& 2^{r-1} ||D ||_1^r + 4^{r-1} \Big[\sup_j \sum_{i\neq j} 
4\min \Big\{ |\sigma_{0,ij}| ,\gamma \Big(\frac{\log p}{n}\Big)^{1/2}\Big\}\Big]^r
 + 4^{r-1} \Big\{4\gamma \Big(\frac{\log p}{n}\Big)^{1/2}\Big\}^r,
\eea
where the third inequality is satisfied by the definition of $A_{ij}$.
Inequality (51) in the proof of Theorem 3 in \cite{cai2012optimal} gives  
\bea
\sum_{i\neq j} 
\min \Big\{ |\sigma_{0,ij}| ,\gamma \Big(\frac{\log p}{n}\Big)^{1/2}\Big\} \le 
C c_{n,p} \Big(\frac{\log p}{n}\Big)^{(1-q)/2},
\eea
for a positive constant $C$.
Thus, we get
\bea
||H_\gamma(\Sigma)-\Sigma_0||_1^r &\le& 2^{r-1} ||D ||_1^r +  4^{r-1}\Big\{4C c_{n,p} \Big(\frac{\log p}{n}\Big)^{(1-q)/2}\Big\}^r + 4^{r-1} \Big\{4\gamma \Big(\frac{\log p}{n}\Big)^{1/2}\Big\}^r.
\eea

Likewise, we have
\bea
||diag^c\{H_\gamma(\Sigma)-\Sigma_0\}||_1^r &\le& 
2^{r-1} ||diag^c(D) ||_1^r + 2^{r-1} ||diag^c[ \{(H_\gamma(\Sigma)_{ij} - \sigma_{0,ij})I(A_{ij})   \}_{1\le i,j\le p}] ||_1^r\\
&\le& 2^{r-1} ||D ||_1^r +  2^{r-1} [\sup_j \sum_{i\neq j} |(H_\gamma(\Sigma)_{ij}- \sigma_{0,ij}|I(A_{ij})]^{r} \\
&\le& 2^{r-1} ||D ||_1^r +2^{r-1}\Big\{4C c_{n,p} \Big(\frac{\log p}{n}\Big)^{(1-q)/2}\Big\}^r .
\eea

\end{proof}

\begin{lemma}\label{lem:concent_element} 
Suppose $\Sigma \in \calG_q(c_{n,p},M_0,M_1)$.
If $x<4 ||\Sigma_0||$ and $(\nu_n -2p) \vee ||A|| = o(n)$, then there exist positive constants $C$ and $\lambda$ such that
\bea
E_{\Sigma_0}\{P^{\pi^i}( |\sigma_{ij} -\sigma_{0,ij}| \ge x  \mid\bbX_n)\} &\le&
C( \exp(-\lambda n)  + \exp(-\lambda n x^2)),
\eea
for arbitrary $i,j\in [p]$. 

\end{lemma}
\begin{proof}

	Since $[\Sigma_{\{i,j\}\times \{i,j\}}|\bbX_n]\sim IW_2(n+\nu_n-2p+4, A_{\{i,j\}\times \{i,j\}}+nS_{\{i,j\}\times \{i,j\}})$ \citep[Theorem 5.2.3]{press2012applied},
			Lemma 6.2 in the supplementary material of \cite{lee2021condmean} gives
	\bea
	&&E_{\Sigma_0}\{P^{\pi^i}( |\sigma_{ij} -\sigma_{0,ij}| \ge x     \mid\bbX_n)\}\\
	&\le& E_{\Sigma_0}\{P^{\pi^i}( || \Sigma_{\{i,j\}\times \{i,j\} } -(\Sigma_0)_{\{i,j\}\times \{i,j\} } ||_2  \ge x   \mid\bbX_n)\}\\
	&\le& C \{\exp(-\lambda n) + \exp(-\lambda nx^2) \},
	\eea
	for some positive constants $C$ and $\lambda$.
\end{proof}

\begin{lemma}\label{lemma:concentD}
Suppose the same notations of Lemma \ref{lem:ineq_thres}. If $(\nu_n -2p) \vee ||A|| \vee \log p= o(n)$ and $\gamma$ is a sufficiently large constant, then there exists a positive constant $C$ such that
\bea
E\{E^{\pi^i}(	||D||_1^2\mid\bbX_n)\} = \frac{C}{n},
\eea
for all sufficiently large $n$. 
\end{lemma}
\begin{proof}
		We have 
\bea
||D||_1^2 &\le& p\sum_{ij} d_{ij}^2 \nonumber\\
&=& p\sum_{ij} (H_\gamma(\Sigma)_{ij}-\sigma_{0,ij})^2 I(A_{ij}^c) 
\{I(|\sigma_{ij}| \ge \gamma (\log p/n)^{1/2} ) + I(|\sigma_{ij}| < \gamma (\log p/n)^{1/2} )  \}\nonumber\\
&\le & p\sum_{ij} (\sigma_{ij}-\sigma_{0,ij})^2 I(A_{ij}^c) I(|\sigma_{ij}| \ge \gamma (\log p/n)^{1/2} ) \\
&&+p\sum_{ij} 
\sigma_{0,ij}^2 I(A_{ij}^c)I(|\sigma_{ij}| < \gamma (\log p/n)^{1/2} ),
\eea
where $\sigma_{ij}$ is the $(i,j)$ element of $\Sigma$.
We also have
\bean
I(A_{ij}^c) I(|\sigma_{ij}| \ge \gamma (\log p/n)^{1/2} ) &=&
I(A_{ij}^c) I(|\sigma_{ij}| \ge \gamma (\log p/n)^{1/2} )  I(|\sigma_{0,ij}| <\frac{\gamma}{2}\Big(\frac{\log p}{n}\Big)^{1/2} ) \nonumber\\
&&+ I(A_{ij}^c) I(|\sigma_{ij}| \ge \gamma (\log p/n)^{1/2} ) I(|\sigma_{0,ij}| \ge\frac{\gamma}{2}\Big(\frac{\log p}{n}\Big)^{1/2} )\nonumber\\
&\le&I(|\sigma_{ij}| \ge \gamma (\log p/n)^{1/2} )  I(|\sigma_{0,ij}| <\frac{\gamma}{2}\Big(\frac{\log p}{n}\Big)^{1/2} ) \nonumber\\
&&+ I(|\sigma_{ij} - \sigma_{0,ij}|> 4\min \Big\{ |\sigma_{0,ij}|,\gamma\sqrt{\frac{\log p}{n}}\Big\})  I(|\sigma_{0,ij}| \ge\frac{\gamma}{2}\Big(\frac{\log p}{n}\Big)^{1/2} )\nonumber\\
&\le&I(|\sigma_{ij}| \ge \gamma (\log p/n)^{1/2} )  I(|\sigma_{0,ij}| <\frac{\gamma}{2}\Big(\frac{\log p}{n}\Big)^{1/2} ) \nonumber\\
&&+ I(|\sigma_{ij}-\sigma_{0,ij}| > 2 \gamma (\log p/n)^{1/2})\nonumber\\
&\le&I(|\sigma_{ij}| -|\sigma_{0,ij}| \ge \frac{\gamma}{2}\Big(\frac{\log p}{n}\Big)^{1/2})+ I(|\sigma_{ij}-\sigma_{0,ij}| > 2 \gamma (\log p/n)^{1/2})\nonumber\\
&\le& I(|\sigma_{ij} -\sigma_{0,ij}| \ge \frac{\gamma}{2}\Big(\frac{\log p}{n}\Big)^{1/2}),\label{formula:Aupper1}
\eean
where the first inequality is satisfied by the definitions of $A_{ij}$ and $H_\gamma(\Sigma)$. 
Next, we have
\bean
I(A_{ij}^c)I(|\sigma_{ij}| < \gamma (\log p/n)^{1/2} ) &=&
I(|\sigma_{ij}| < \gamma (\log p/n)^{1/2} )I( |\sigma_{0,ij}| > 4\min \Big\{ |\sigma_{0,ij}|,\gamma\sqrt{\frac{\log p}{n}}\Big\}   )\nonumber\\
&\le & I(|\sigma_{0,ij}|-|\sigma_{ij}- \sigma_{0,ij}| \le \gamma (\log p/n)^{1/2} ) I(|\sigma_{0,ij}| > 4\gamma\sqrt{\frac{\log p}{n}})\nonumber\\
&\le& I(|\sigma_{ij}- \sigma_{0,ij}| \ge 3|\sigma_{0,ij}|/4 ) I(|\sigma_{0,ij}| > 4\gamma\sqrt{\frac{\log p}{n}}).\label{formula:Aupper2}
\eean
Using the inequalties \eqref{formula:Aupper1} and \eqref{formula:Aupper2}, we get 
\bean
E_{\Sigma_0}\{ E^{\pi^i}(||D||_1^2\mid\bbX_n)\} &\le& 
p\sum_{ij} E\{ E^{\pi^i}( (\sigma_{ij}-\sigma_{0,ij})^2 I(|\sigma_{ij} -\sigma_{0,ij}| \ge \frac{\gamma}{2}\Big(\frac{\log p}{n}\Big)^{1/2})   \mid\bbX_n)\}		\label{formula:sparseupperbound1}\\
&&+p\sum_{ij} E\{ E^{\pi^i}(  \sigma_{0,ij}^2 I(|\sigma_{ij}- \sigma_{0,ij}| \ge \frac{3}{4}|\sigma_{0,ij}| ) I(|\sigma_{0,ij}| > 4\gamma\sqrt{\frac{\log p}{n}})  \mid\bbX_n)\}.\label{formula:sparseupperbound2}
\eean
For the upper bound of \eqref{formula:sparseupperbound1}, we have
\bea
\eqref{formula:sparseupperbound1}\le p^3  \sup_{ij} E_{\Sigma_0}\{ E^{\pi^i}(( \sigma_{ij}-\sigma_{0,ij})^4\mid\bbX_n)\}^{1/2}  E_{\Sigma_0}\{P^{\pi^i}( |\sigma_{ij} -\sigma_{0,ij}| \ge \frac{\gamma}{2}\Big(\frac{\log p}{n}\Big)^{1/2}     \mid\bbX_n)\}^{1/2},
\eea
and 
\bea
E_{\Sigma_0}\{ E^{\pi^i}(( \sigma_{ij}-\sigma_{0,ij})^4\mid\bbX_n)\} \le 
E_{\Sigma_0}\{ E^{\pi^i}(|| \Sigma_{\{i,j\}\times \{i,j\} } -(\Sigma_0)_{\{i,j\}\times \{i,j\} } ||_2^4\mid\bbX_n)\}.
\eea
Since $[\Sigma_{\{i,j\}\times \{i,j\}}|\bbX_n]\sim IW_2(n+\nu_n-2p+4, A_{\{i,j\}\times \{i,j\}}+nS_{\{i,j\}\times \{i,j\}})$ \citep[Theorem 5.2.3]{press2012applied}, 
Lemma 6.1 in the supplementary material of \cite{lee2021condmean} gives that there exist a positive constant $C_1$ such that
\bean
E_{\Sigma_0}\{ E^{\pi^i}(|| \Sigma_{\{i,j\}\times \{i,j\} } -(\Sigma_0)_{\{i,j\}\times \{i,j\} } ||_2^4\mid\bbX_n)\}^{1/2} \le 
\frac{C_1}{n},\label{formula:sparseupper3}
\eean
for all sufficiently large $n$. 
Lemma \ref{lem:concent_element} gives that there exist some positive constants $C_2$ and $\lambda_1$ such that
\bea
E_{\Sigma_0}\{P^{\pi^i}( |\sigma_{ij} -\sigma_{0,ij}| \ge \frac{\gamma}{2}\Big(\frac{\log p}{n}\Big)^{1/2}     \mid\bbX_n)\}^{1/2}
&\le& C_2 \{\exp(-\lambda_1 n) + \exp(-\lambda_1\gamma^2 \log p /2) \},
\eea
for all sufficiently large $n$. 
Thus, we get 
\bea
\eqref{formula:sparseupperbound1} \le C_3 \frac{1}{n}(\frac{1}{  p^{\lambda_1 \gamma^2/2-3}} + p^3\exp(-\lambda_1 n/2)   ),
\eea
for a positive constant $C_3$.
If $\gamma$ is a sufficiently large constant and $\log p = o(n)$, then $\eqref{formula:sparseupperbound1} = O(1/n)$.

For the upper bound of \eqref{formula:sparseupperbound2}, there exist positive constants $C_4$ and $\lambda_2$ such that
\bean
\eqref{formula:sparseupperbound2} 
&\le& \sum_{ij} \sigma_{0,ij}^2 I( |\sigma_{0,ij}| \ge 4\gamma (\log p/n)^{1/2})  E_{\Sigma_0} \{P^{\pi^i} ( |\sigma_{ij} - \sigma_{0,ij}| > \frac{3}{4}|\sigma_{0,ij}| \mid\bbX_n) \} \nonumber\\
&\le& \sum_{ij} \sigma_{0,ij}^2 I( |\sigma_{0,ij}| \ge 4\gamma (\log p/n)^{1/2})  E_{\Sigma_0} \{P^{\pi^i} ( || \Sigma_{\{i,j\}\times \{i,j\} } -(\Sigma_0)_{\{i,j\}\times \{i,j\} } ||_2 > \frac{3}{4}|\sigma_{0,ij}| \mid\bbX_n) \} \nonumber\\
&\le& C_4\frac{p^2}{n} \sup_{ij} \{n\sigma_{0,ij}^2 \exp(- n\lambda_2 \sigma_{0,ij}^2 ) I( n\sigma_{0,ij}^2 \ge 16\gamma^2 \log p)\}\nonumber\\
&\le& C_4\frac{p^2}{n} 16\gamma^2 \log p\exp(-\lambda_2 16\gamma^2 \log p)\nonumber\\
&\le& \frac{C_4}{n}\exp(-\lambda_2 16\gamma^2 \log p + 2\log p + \log\log p),\label{formula:D2}
\eean
for all sufficiently large $n$.
The third inequality holds by Lemma 6.2 in the supplementary material of \cite{lee2021condmean}. 
If $\gamma$ is a sufficiently large constant and $\log p=o(n)$, then $\eqref{formula:sparseupperbound2} = O(1/n)$.

\end{proof}

	\begin{proof}[Proof of Theorem \ref{main-thm:upperbound}] 
	    By Lemma 2.1 in the supplementary material of \cite{lee2021condmean}, we have
	    \bea
	    ||H^{(\epsilon_n)}_\gamma(\Sigma)-\Sigma_0||_2^2 &\le&
	    2^3 ||H_\gamma(\Sigma)-\Sigma_0||_2^2 + 4\epsilon_n^2.
	    \eea
	    Lemma \ref{lem:ineq_thres} gives that there exist a positive constant $C_1$ such that
	    \bea
		||H_\gamma(\Sigma)-\Sigma_0||_2^2 
		&\le& ||H_\gamma(\Sigma)-\Sigma_0||_1^2 \\
		&\le&  C_1(c_{n,p}^2\Big(\frac{\log p}{n}\Big)^{1-q}+\frac{\log p}{n}  )+ 2 ||D||_1^2,
		\eea
		where $D=(d_{ij})_{1\le i,j\le p}$, $d_{ij}=(H_\gamma(\Sigma)_{ij}-\sigma_{0,ij}) I(A_{ij}^c)$, and 
\bea
A_{ij}=\Big\{ |(H_\gamma(\Sigma))_{ij}-\sigma_{0,ij}|\le 4\min \Big\{ |\sigma_{0,ij}|,\gamma\sqrt{\frac{\log p}{n}}\Big\} \Big\}.
\eea
By applying Lemma \ref{lemma:concentD} to the term $||D_1||^2$, we complete the proof.

	\end{proof}
\se{Proof of Theorem \ref{main-thm:portfolio}}
In this section, we prove Theorem \ref{main-thm:portfolio} to show the upper bound of the post-processed posterior in the respect of the global minimum variance portfolio. Before proving the theorem, we present two lemmas and two corollaries from one of the lemma.

\begin{lemma}\label{lemma:1normconcent}
	Let $M$ be an arbitrary positive constant. If $(\nu_n -2p) \vee ||A||= o(n)$  and $x$ satisfies 
	\bea
	\Big[x- \Big\{C_1 c_{n,p} \Big(\frac{\log p}{n}\Big)^{(1-q)/2} + 
	4\gamma \Big(\frac{\log p}{n}\Big)^{1/2} \Big\}\Big] > M,
	\eea
	for a positive constant $C_1$, 
	then there exist positive constants $C$, $\lambda$ such that
\bea
P(||H_\gamma(\Sigma) - \Sigma_0||_1 > x) \le C (2p^2\exp(-\lambda n)  +p^2\exp(-\lambda n^\delta )+ n^{(2-\lambda \gamma)(1-\delta)/2}),
\eea
for all $\gamma >2/\lambda$, all sufficiently large $n$ and all $\delta\in(0,1)$.
\end{lemma}
\begin{proof}
	
Using Lemma \ref{lem:ineq_thres}, we have
\bea
P^{\pi^i}(||H_\gamma(\Sigma) - \Sigma_0||_1 > x\mid\bbX_n) &\le& 
P^{\pi^i}(||D||_1 > x -  s_n \mid\bbX_n) \\
&\le&p\sup_j P^{\pi^i}( \sum_{i=1}^p  |d_{ij}| > x - s_n  \mid\bbX_n)\\
    &\le& p^2 \sup_{i,j}P^{\pi^i}(|d_{ij}| >(x-s_n)/p\mid\bbX_n),
\eea
where $s_n =  C_1 c_{n,p} ((\log p)/n)^{(1-q)/2} + 
4\gamma ((\log p)/n)^{1/2}$ for a positive constant $C_1$.
We show the upper bound of $\sup_{i,j}P^{\pi^i}(|d_{ij}| >(x-s_n)/p\mid\bbX_n)$ for the two cases: 
$p \le n^{(1-\delta)/2}$ and $p > n^{(1-\delta)/2}$, where $\delta$ is an arbitrary real number in $(0,1)$.
First, we consider the case $p \le n^{(1-\delta)/2}$.
 We have 
\bea
&&P^{\pi^i}(|d_{ij}| > (x-s_n)/p\mid\bbX_n) \\
&\le& P^{\pi^i}(|H_\gamma(\Sigma)_{ij}-\sigma_{0,ij} | >(x-s_n)/p\mid\bbX_n)\\
&\le&
P^{\pi^i}(|\sigma_{ij}-\sigma_{0,ij} | >(x-s_n)/p\mid\bbX_n) +
E^{\pi^i}(I(|\sigma_{0,ij} | >(x-s_n)/p) I(|\sigma_{ij}|<\gamma ((\log p)/n)^{1/2})\mid\bbX_n)\\
&\le&
P^{\pi^i}(|\sigma_{ij}-\sigma_{0,ij} | >(x-s_n)/p\mid\bbX_n) + 
P^{\pi^i}(|\sigma_{0,ij} |- |\sigma_{ij}| > (x-s_n)/p - \gamma ((\log p)/n)^{1/2}\mid\bbX_n)\\
&\le& 2 P^{\pi^i}(|\sigma_{ij}-\sigma_{0,ij} | >(x-s_n)/p- \gamma (\log p/n)^{1/2}\mid\bbX_n)\\
&\le& 2 P^{\pi^i}(|\sigma_{ij}-\sigma_{0,ij} | >(x-s_n)n^{(\delta-1)/2} - \gamma ( (1-\delta)(\log n)/(2 n))^{1/2}\mid\bbX_n),
\eea
where the assumption $p\le n^{(1-\delta)}$ is used only in the last inequality.
For all sufficiently large $n$, we have
\bea
(x-s_n)n^{(\delta-1)/2} - \gamma ( (1-\delta)(\log n)/(2 n))^{1/2} > Mn^{(\delta-1)/2}/2 .
\eea Thus, there exist some positive constants $C_2$ and $\lambda_1$ such that 
\bean
P^{\pi^i}(|d_{ij}| > (x-s_n)/p\mid\bbX_n) &\le& C_2[\exp(-n\lambda_1)  +  \exp(-\lambda_1 n^\delta )],\label{formula:concent_res1}
\eean
for all sufficiently large $n$ by Lemma \ref{lem:concent_element}.

 Next we consider the second case $p > n^{(1-\delta)/2}$. 
Since $x-s_n>0$, we have 
 \bea
 P^{\pi^i}(|d_{ij}| > (x-s_n)/p \mid\bbX_n) &\le& P^{\pi^i}(A_{ij}^c\mid\bbX_n).
 \eea
Note that  
\bea
I(A_{ij}^c)  &=& I(A_{ij}^c) \{ I(|\sigma_{ij}| \ge \gamma (\log p/n)^{1/2} )  + I(|\sigma_{ij}| < \gamma (\log p/n)^{1/2} )    \} \\
&\le&  I(|\sigma_{ij} -\sigma_{0,ij}| \ge \frac{\gamma}{2}\Big(\frac{\log p}{n}\Big)^{1/2}) +   I(|\sigma_{ij} -\sigma_{0,ij}| \ge 3\gamma \Big(\frac{\log p}{n}\Big)^{1/2}),
\eea
by \eqref{formula:Aupper1} and \eqref{formula:Aupper2}.
Thus, by Lemma \ref{lem:concent_element}, there exist some positive constants $C_3$ and $\lambda_2$ such that 
\bean
P^{\pi^i}(A_{ij}^c\mid\bbX_n) \le  C_3(\exp(-\lambda_2 n) + \exp(-\lambda_2 \gamma \log p))\label{formula:probAij}.
\eean	
If $\gamma$ is a sufficiently large positive constant such that $\gamma > 2/\lambda_2$, then
\bean
p^2 \sup_{i,j}P^{\pi^i}(A_{ij}^c\mid\bbX_n)&\le&   C_3 (p^2\exp(-\lambda_2 n)  + p^{2-\lambda_2 \gamma} )\nonumber\\
&\le&   C_3 (p^2\exp(-\lambda_2 n)  + n^{(2-\lambda_2\gamma)(1-\delta)/2} ).\label{formula:concent_res2}
\eean
The assumption $p > n^{(1-\delta)/2}$ is used only in the last inequality.
By combining \eqref{formula:concent_res1} and \eqref{formula:concent_res2}, the proof is completed.

\end{proof}

Using Lemma \ref{lemma:1normconcent}, we derive two Corollaries \ref{cor:maxevconcent} and \ref{cor:minevconcent}.
\begin{corollary}\label{cor:maxevconcent}
Suppose the same setting of Lemma \ref{lemma:1normconcent}. If $$x > 2||\Sigma_0||   +\Big\{C_1 c_{n,p} \Big(\frac{\log p}{n}\Big)^{(1-q)/2} + 
4\gamma \Big(\frac{\log p}{n}\Big)^{1/2} \Big\}$$ for a positive constant $C_1$, then there exist positive constants $C$, $\lambda$ such that
\bea
P^{\pi^i}(\lambda_{\max}(H_\gamma(\Sigma))  > x) \le C (2p^2\exp(-\lambda n)  +p^2\exp(-\lambda n^\delta )+ n^{(2-\lambda \gamma)(1-\delta)/2}),
\eea
for all $\gamma >2/\lambda$, all sufficiently large $n$ and all $\delta\in(0,1)$.

\end{corollary}
\begin{proof}
	We have 
\bea
P^{\pi^i}(\lambda_{\max}(H_\gamma(\Sigma)) > x \mid\bbX_n)  &\le& 
P^{\pi^i}( ||H_\gamma(\Sigma)-\Sigma_0|| > x-||\Sigma_0||\mid\bbX_n) \\
&\le&P^{\pi^i}( ||H_\gamma(\Sigma)-\Sigma_0||_1 > x-||\Sigma_0||\mid\bbX_n) 
\eea	
Since $x-||\Sigma_0|| > ||\Sigma_0||  +\Big\{C_1 c_{n,p} \Big(\frac{\log p}{n}\Big)^{(1-q)/2} + 
4\gamma \Big(\frac{\log p}{n}\Big)^{1/2} \Big\}$, we apply Lemma \ref{lemma:1normconcent} and the proof is completed.

\end{proof}

\begin{corollary}\label{cor:minevconcent}
Suppose the same setting of Lemma \ref{lemma:1normconcent}. If 
	\bea
	 x <  \lambda_{\min}(\Sigma_0)/2 -\Big\{C_1 c_{n,p} \Big(\frac{\log p}{n}\Big)^{(1-q)/2} + 
	 4\gamma \Big(\frac{\log p}{n}\Big)^{1/2} \Big\},
	\eea
	for a positive constant $C_1$,
	then there exist positive constants $C$, $\lambda$ such that
	\bea
	P^{\pi^i}(\lambda_{\min}(H_\gamma(\Sigma)) < x\mid\bbX_n) \le C (2p^2\exp(-\lambda n)  +p^2\exp(-\lambda n^\delta )+ n^{(2-\lambda \gamma)(1-\delta)/2}),
	\eea
	for all $\gamma >2/\lambda$, all sufficiently large $n$ and all $\delta\in(0,1)$.
\end{corollary}
\begin{proof}
	We have
\bea
	P^{\pi^i}(\lambda_{\min}(H_\gamma(\Sigma)) < x\mid\bbX_n)  &\le&
	P^{\pi^i}(\lambda_{\min}(\Sigma_0) -||H_\gamma(\Sigma)-\Sigma_0|| < x \mid\bbX_n)\\
	 &\le& P^{\pi^i}(||H_\gamma(\Sigma)-\Sigma_0||_1 >\lambda_{\min}(\Sigma_0)- x \mid\bbX_n)
\eea	
Since $\lambda_{\min}(\Sigma_0)-x > \lambda_{\min}(\Sigma_0)/2   +\Big\{C_1 c_{n,p} \Big(\frac{\log p}{n}\Big)^{(1-q)/2} + 
4\gamma \Big(\frac{\log p}{n}\Big)^{1/2} \Big\}$, the proof is completed by Lemma \ref{lemma:1normconcent}.

\end{proof}

\begin{lemma}\label{lem:4thmoment}
If $c_{n,p}^4 (\log p/n)^{2(1-q)} + \gamma^4((\log p)/n)^2$ is bounded and $(\nu_n -2p) \vee ||A||= o(n)$, then 
$E_{\Sigma_{0}}E^{\pi^i} (||(\Sigma_{0}-H_\gamma^{(\epsilon_n)}(\Sigma) )\Sigma_{0}^{-1}\mathbf{1} ||^4 \mid\bbX_n)$ is bounded by a positive constant for all sufficiently large $p$. 
\end{lemma}
\begin{proof}
    By Lemma \ref{lem:ineq_thres}, it suffices to show that $E_{\Sigma_0}\{E^{\pi^i}(||D||_1^4\mid\bbX_n)\}$ is bounded.
We have
\bea
E_{\Sigma_0}\{E^{\pi^i}(||D||_2^4\mid\bbX_n)\} &\le&   E_{\Sigma_0}\{E^{\pi^i}(\sup_i (\sum_j|d_{ij}|)^4\mid\bbX_n)\}\\
&\le& p^3  \sum_{i,j}E_{\Sigma_0}\{E^{\pi^i}(|d_{ij}|^4\mid\bbX_n)\}\\
&\le& p^5  \sup_{i,j}E_{\Sigma_0}\{E^{\pi^i}(| H_\gamma(\Sigma)_{ij} - \sigma_{0,ij}|^4 I(A_{ij}^c)  \mid\bbX_n)\}\\
&\le& 
p^5  \sup_{i,j} E_{\Sigma_0}\{E^{\pi^i}( |\sigma_{ij}-\sigma_{0,ij}|^4I(A_{ij}^c)  \mid\bbX_n)\} \\
&&+p^5  \sup_{i,j} E_{\Sigma_0}\{E^{\pi^i}( |\sigma_{0,ij}|^4I(A_{ij}^c) I(|\sigma_{ij}< \gamma(\log p/n)^{1/2})  \mid\bbX_n)\}
\eea
By \eqref{formula:probAij} and Lemma 6.1 in the supplementary material of \cite{lee2021condmean}, there exist positive constants $C_1$ and $\lambda_1$ such that
\bea
E_{\Sigma_0}\{E^{\pi^i}( |\sigma_{ij}-\sigma_{0,ij}|^4I(A_{ij}^c)  \mid\bbX_n)\} &\le& 
[ E\{( \sigma_{ij}-\sigma_{0,ij})^8 \mid \bbX_n\} ]^{1/2}  P(A_{ij}^c\mid\bbX_n)^{1/2}\\
&\le&\{ E( || \Sigma_{\{i,j\}\times \{i,j\} } -(\Sigma_0)_{\{i,j\}\times \{i,j\} } ||_2^8 \mid \bbX_n) \}^{1/2}  P(A_{ij}\mid\bbX_n)^{1/2} \\
&\le& C_1\frac{p^2}{n^2}  p^{-\lambda_1 \gamma}.
\eea
Next, we have 
\bea
&& E_{\Sigma_0}\{E^{\pi^i}( |\sigma_{0,ij}|^4I(A_{ij}^c) I(|\sigma_{ij}< \gamma(\log p/n)^{1/2})  \mid\bbX_n)\} \\
 &\le&  E_{\Sigma_0}\{E^{\pi^i}( |\sigma_{0,ij}|^4
  I(|\sigma_{ij}- \sigma_{0,ij}| \ge 3|\sigma_{0,ij}|/4 ) I(|\sigma_{0,ij}| > 4\gamma\sqrt{\frac{\log p}{n}}) \mid\bbX_n)\}\\
  &\le& |\sigma_{0,ij}|^4I(|\sigma_{0,ij}| > 4\gamma\sqrt{\frac{\log p}{n}}) 
   E_{\Sigma_0}\{P^{\pi^i}(  ||\Sigma_{\{i,j\}\times \{i,j\}} - (\Sigma_0)_{\{i,j\}\times \{i,j\}} >3|\sigma_{0,ij}|/4 \mid\bbX_n)\}\\
   &\le& |\sigma_{0,ij}|^4I(|\sigma_{0,ij}| > 4\gamma\sqrt{\frac{\log p}{n}})  \exp(-n\lambda_1 \sigma_{0,ij}^2) \\
   &= &\frac{1}{n^2}  (n\sigma_{0,ij}^2)^2 \exp(-n\lambda_1 \sigma_{0,ij}^2 )
  	I(n\sigma_{0,ij}^2 > 16\gamma^2 \log p),
\eea
where the first inequality is satisfied by \eqref{formula:Aupper2}.
Note that $f(x) = x^2\exp(-\lambda_1 x)$ is a decreasing function when $x>2/\lambda_1$. 
If $p$ is large enough to satisfy $16\gamma^2 \log p >2/\lambda_1$, then 
\bea
&&\frac{1}{n^2}  (n\sigma_{0,ij}^2)^2 \exp(-n\lambda_1 \sigma_{0,ij}^2 )
I(n\sigma_{0,ij}^2 > 16\gamma^2 \log p)\\
&\le& \frac{1}{n^2} ( 16\gamma^2 \log p)^2 \exp ( -\lambda_1  16\gamma^2 \log p)\\
&=& \frac{( 16\gamma^2 \log p)^2 }{ n^2   p^{\lambda_1  16\gamma^2}}.
\eea
Thus, 
\bea
E_{\Sigma_0}\{E^{\pi^i}(||D||_2^4\mid\bbX_n)\}  \le \frac{p^{5-\lambda_1 \gamma }}{n^2} + \frac{( 16\gamma^2 \log p)^2 }{ n^2   p^{\lambda_1  16\gamma^2}},
\eea
which is bounded by a positive constant when $\gamma$ is a sufficiently large constant.

\end{proof}

 Using the lemmas and corollaries, we prove Theorem \ref{main-thm:portfolio}.
    \begin{proof}[Proof of Theorem \ref{main-thm:portfolio}]	
    
    We have 
    \bea
    \Bigg|\Bigg|\frac{H_\gamma^{(\epsilon_n)}(\Sigma)^{-1}\mathbf{1}}{\mathbf{1}^TH_\gamma^{(\epsilon_n)}(\Sigma)^{-1}\mathbf{1}} -  \frac{\Sigma_0^{-1}\mathbf{1}}{\mathbf{1}^T\Sigma_0^{-1}\mathbf{1}} \Bigg|\Bigg|_2^2
    &\le&  
   	\frac{2}{|\mathbf{1}^TH_\gamma^{(\epsilon_n)}(\Sigma)^{-1}\mathbf{1}|^2} 
   	||(H_\gamma^{(\epsilon_n)}(\Sigma)^{-1} -\Sigma_{0}^{-1})\mathbf{1}||_2^2 \\
   	&&+ \frac{2 ||\Sigma_{0}^{-1}\mathbf{1}/\sqrt{p}||^2}{|(\mathbf{1}/\sqrt{p})^T\Sigma_{0}^{-1}\mathbf{1}/\sqrt{p}|^2} 
   	\frac{|(\mathbf{1}/\sqrt{p})^T (H_\gamma^{(\epsilon_n)}(\Sigma)^{-1} -\Sigma_{0}^{-1})\mathbf{1}|^2}{|\mathbf{1}^TH_\gamma^{(\epsilon_n)}(\Sigma)^{-1}\mathbf{1}|^2}\\
   	&\le& 	\frac{2   	||(H_\gamma^{(\epsilon_n)}(\Sigma)^{-1} -\Sigma_{0}^{-1})\mathbf{1}||_2^2}{|\mathbf{1}^TH_\gamma^{(\epsilon_n)}(\Sigma)^{-1}\mathbf{1}|^2}  \Big( 1+  \frac{ ||\Sigma_{0}^{-1}||^2}{\lambda_{\min}(\Sigma_{0}^{-1})^2} \Big) \\
   	&\le& 	\frac{2}{p^2}   	||(H_\gamma^{(\epsilon_n)}(\Sigma)^{-1} -\Sigma_{0}^{-1})\mathbf{1}||_2^2   \lambda_{\max}(H_\gamma^{(\epsilon_n)}(\Sigma))^2  \Big( 1+  \frac{ ||\Sigma_{0}^{-1}||^2}{\lambda_{\min}(\Sigma_{0}^{-1})^2} \Big)\\
   	&\le& 	\frac{2}{p^2}   	||(\Sigma_{0}-H_\gamma^{(\epsilon_n)}(\Sigma) )\Sigma_{0}^{-1}\mathbf{1}||_2^2  ||H_\gamma^{(\epsilon_n)}(\Sigma)^{-1}||^2 ||H_\gamma^{(\epsilon_n)}(\Sigma)||^2  \Big( 1+  \frac{ ||\Sigma_{0}^{-1}||^2}{\lambda_{\min}(\Sigma_{0}^{-1})^2} \Big).
    \eea
    Next, for an arbitrary positive constant $C_1$, we have
    \bean
 &&   E^{\pi^i} (	||(\Sigma_{0}-H_\gamma^{(\epsilon_n)}(\Sigma) )\Sigma_{0}^{-1}\mathbf{1}||_2^2  ||H_\gamma^{(\epsilon_n)}(\Sigma)^{-1}||^2 ||H_\gamma^{(\epsilon_n)}(\Sigma)||^2 \mid\bbX_n)\nonumber\\
    &\le&C_1^4 E^{\pi^i} ( 	||(\Sigma_{0}-H_\gamma^{(\epsilon_n)}(\Sigma) )\Sigma_{0}^{-1}\mathbf{1}||_2^2 \mid\bbX_n) \label{formula:portupper1}\\
    &&+ E^{\pi^i} (||(\Sigma_{0}-H_\gamma^{(\epsilon_n)}(\Sigma) )\Sigma_{0}^{-1}\mathbf{1} ||^4 \mid\bbX_n)^{1/2}  E^{\pi^i}\{I(||H_\gamma^{(\epsilon_n)}(\Sigma)^{-1}||>C_1  ) I(||H_\gamma^{(\epsilon_n)}(\Sigma)||>C_1)\mid\bbX_n\}^{1/2}.\label{formula:portupper2}
    \eean
    We show the upper bound of \eqref{formula:portupper2}. 
    Lemma \ref{lem:4thmoment} shows that $E^{\pi^i} (||(\Sigma_{0}-H_\gamma^{(\epsilon_n)}(\Sigma) )\Sigma_{0}^{-1}\mathbf{1} ||^4 \mid\bbX_n)^{1/2} $ is bounded by a positive constant. 
    There exist positive constants $C_2$ and $\lambda_1$ such that
    \bea
    &&E^{\pi^i}\{I(||H_\gamma^{(\epsilon_n)}(\Sigma)^{-1}||>C_1  ) I(||H_\gamma^{(\epsilon_n)}(\Sigma)||>C_1)\mid\bbX_n\}\\
    &\le& E^{\pi^i}\{I(||H_\gamma(\Sigma)^{-1}||>C_1  ) I(||H_\gamma(\Sigma)||>C_1)\mid\bbX_n\} \\
    &&+ E^{\pi^i}\{I(\lambda_{\min}(H_\gamma(\Sigma))< \epsilon_n  ) \mid\bbX_n\} \\
    &\le& P^{\pi^i} (||H_\gamma(\Sigma)||>C_1\mid\bbX_n)^{1/2}
    P^{\pi^i} (\lambda_{\min}(H_\gamma(\Sigma))<C_1^{-1}\mid\bbX_n)^{1/2}\\
    &&+ E^{\pi^i}\{I(\lambda_{\min}(H_\gamma(\Sigma))< \epsilon_n  ) \mid\bbX_n\} \\
  	 &\le& P^{\pi^i} (||H_\gamma(\Sigma)||>C_1\mid\bbX_n)^{1/2}
  	P^{\pi^i} (\lambda_{\min}(H_\gamma(\Sigma))<C_1^{-1}\mid\bbX_n)^{1/2}\\
     &&+ P^{\pi^i}(\lambda_{\min}(\Sigma_{0})- \epsilon_n < ||H_\gamma(\Sigma)-\Sigma_{0}|| \mid\bbX_n )  \\
    &\le& C_2 (p^2\exp(-\lambda_1 n)  +p^2\exp(-\lambda_1 n^\delta )+ n^{(2-\lambda_1 \gamma)(1-\delta)/2}),
        \eea
     for all sufficiently large $n$.
    The last inequality is satisfied by Corollarys \ref{cor:maxevconcent}, \ref{cor:minevconcent} and Lemma \ref{lemma:1normconcent} since $\epsilon_n<\lambda_{\min}(\Sigma_{0})/2$ and $c_{n,p}((\log p)/n)^{(1-q)/2} + 4\gamma ((\log p)/n)^{1/2} \lra 0$ as $n\lra \infty$. 
    
    We show the upper bound of \eqref{formula:portupper1}. There exists positive constant $C_3$ and $\lambda_2$ such that
    \bea
     && E_{\Sigma_0} E^{\pi^i}( ||(\Sigma_{0}-H_\gamma^{(\epsilon_n)}(\Sigma) )\Sigma_{0}^{-1}\mathbf{1} ||^2 \mid\bbX_n) \\&\le&
     E_{\Sigma_0} E^{\pi^i}(   ||(\Sigma_{0}-H_\gamma(\Sigma) )\Sigma_{0}^{-1}\mathbf{1} ||^2\mid\bbX_n)\\
        &&+E_{\Sigma_0} E^{\pi^i}(||(\Sigma_{0}-H_\gamma^{(\epsilon_n)}(\Sigma) ) ||^2I(\lambda_{\min}(H_\gamma(\Sigma))<\epsilon_n) \mid\bbX_n)  ||\Sigma_{0}^{-1}\mathbf{1} ||^2\\
         &\le& 2E_{\Sigma_0} E^{\pi^i}(|| diag(\Sigma_0 - H_\gamma(\Sigma)) \Sigma_{0}^{-1}\mathbf{1} ||^2 + 
        2||diag^c(\Sigma_0 - H_\gamma(\Sigma)) \Sigma_{0}^{-1}\mathbf{1} ||^2\mid\bbX_n)\\
        &&+||\Sigma_{0}^{-1}\mathbf{1} ||^2E_{\Sigma_0} E^{\pi^i}(||(\Sigma_{0}-H_\gamma^{(\epsilon_n)}(\Sigma) ) ||^4\mid\bbX_n )^{1/2}
      E_{\Sigma_0}  P^{\pi^i}(\lambda_{\min}(\Sigma_{0})- \epsilon_n < ||H_\gamma(\Sigma)-\Sigma_{0}|| \mid\bbX_n )^{1/2}\\
        &\le&  2E_{\Sigma_0} E^{\pi^i}(|| diag(\Sigma_0 - H_\gamma(\Sigma)) \Sigma_{0}^{-1}\mathbf{1} ||^2 + 
        2||diag^c(\Sigma_0 - H_\gamma(\Sigma)) \Sigma_{0}^{-1}\mathbf{1} ||^2\mid\bbX_n)\\
&&+ C_3  p (p^2\exp(-\lambda_2 n)  +p^2\exp(-\lambda_2 n^\delta )+ n^{(2-\lambda_2 \gamma)(1-\delta)/2}),
    \eea
    for all sufficiently large $n$. The last inequality is satisfied by Lemma \ref{lemma:1normconcent}.
    Next, we have 
    \bea
    &&E_{\Sigma_0} E^{\pi^i}(|| diag(\Sigma_0 - H_\gamma(\Sigma)) \Sigma_{0}^{-1}\mathbf{1} ||^2\mid\bbX_n)\\
    &\le&  ||\Sigma_0^{-1}||^2 \sum_{i=1}^p E_{\Sigma_0} E^{\pi^i}((\Sigma_{0,ii} -H_\gamma(\Sigma)_{ii})^2\mid\bbX_n)\\
    &\le& 2 ||\Sigma_0^{-1}||^2   \sum_{i=1}^{p-1} E_{\Sigma_0} E^{\pi^i}(||(\Sigma_0)_{i:i+1,i:i+1} - H_\gamma(\Sigma)_{i:i+1,i:i+1}||^2 \mid\bbX_n)
    \eea
    We apply Theorem \ref{main-thm:upperbound} considering that the true covariance is $(\Sigma_0)_{i:i+1,i:i+1}$. There exists a positive constant $C_4$ such that
   \bea
	   \sum_{i=1}^{p-1}E_{\Sigma_0} E^{\pi^i}( ||(\Sigma_0)_{i:i+1,i:i+1} - H_\gamma(\Sigma)_{i:i+1,i:i+1}||^2 \mid\bbX_n) \le 
	  C_4 \frac{p}{n},
   \eea
   for all sufficiently large $n$. 
    Lemmas \ref{lem:ineq_thres} and \ref{lemma:concentD} give that there exist a positive constant $C_4$ such that
    \bea
 &&  E_{\Sigma_0} E^{\pi^i}(||diag^c(\Sigma_0 - H_\gamma(\Sigma)) \Sigma_{0}^{-1}\mathbf{1} ||^2 \mid\bbX_n)\\
 &\le&
 ||  \Sigma_{0}^{-1}\mathbf{1} ||^2 E_{\Sigma_0} E^{\pi^i}\Big(2 || D||_1^2  + 2\Big\{4C c_{n,p} \Big(\frac{\log p}{n}\Big)^{(1-q)/2}\Big\}^2\mid\bbX_n\Big) \\
 &\le&  pC_4 (\frac{1}{n} + c_{n,p}^2 \Big(\frac{\log p}{n}\Big)^{(1-q)}),
    \eea
    for all sufficiently large $n$. 
    
    Collecting all the inequalities above, we have 
    \bea
        &&E_{\Sigma_0}  E^{\pi^i}  (||(\Sigma_{0}-H_\gamma^{(\epsilon_n)}(\Sigma) )\Sigma_{0}^{-1}\mathbf{1} ||^2 \mid\bbX_n) \\
        &\le& \frac{1}{p}   (p^2\exp(-\lambda_2 n)  +p^2\exp(-\lambda_2 n^\delta )+ n^{(2-\lambda_2 \gamma)(1-\delta)/2} + \frac{1}{n} + c_{n,p}^2 \Big(\frac{\log p}{n}\Big)^{(1-q)}),
    \eea
    and
    \bea
   E_{\Sigma_0}  E^{\pi^i} \Bigg(	 \Bigg|\Bigg|\frac{H_\gamma^{(\epsilon_n)}(\Sigma)^{-1}\mathbf{1}}{\mathbf{1}^TH_\gamma^{(\epsilon_n)}(\Sigma)^{-1}\mathbf{1}} -  \frac{\Sigma_0^{-1}\mathbf{1}}{\mathbf{1}^T\Sigma_0^{-1}\mathbf{1}} \Bigg|\Bigg|_2^2\mid\bbX_n\Bigg) \lesssim \frac{1}{p} \Big(  \frac{1}{n} + c_{n,p}^2 \Big(\frac{\log p}{n}\Big)^{(1-q)}  \Big).
    \eea  
    Since 
    \bea
  \Bigg|\Bigg|\frac{\Sigma_0^{-1}\mathbf{1}}{\mathbf{1}^T\Sigma_0^{-1}\mathbf{1}} \Bigg|\Bigg|^2\le \frac{||\Sigma_0^{-1}||^2 ||\Sigma_0||^2}{p} ,
  \eea
  the proof is completed. 

        \end{proof}

\se{Proof of Lemmas \ref{main-lemma:lowerbound1} and \ref{main-lemma:lowerbound2}}
For the proof of Lemma \ref{main-lemma:lowerbound1}, we present Lemma \ref{lem:lower1}.

\begin{lemma}\label{lem:lower1}
	Suppose the notation of $\calG^{(1)}$ in \ref{main-calG1}.
	If $\tau/\sqrt{n} \le M/3$ and $\tau/M \le 1/3$, then there exists a positive constant $C$ such that
	\bea
	\min_{\{(\theta,\theta')\in\{0,1\}^{p'}:H(\theta,\theta')=1\}} ||\mathbb{P}_\theta \wedge \mathbb{P}_{\theta'} || \ge C.
	\eea
\end{lemma}
\begin{proof}
	Since $||\mathbb{P}_\theta \wedge \mathbb{P}_{\theta'} ||  = 1- 0.5|| \mathbb{P}_\theta -  \mathbb{P}_{\theta'} ||_1$, it suffices to show $|| \mathbb{P}_\theta -  \mathbb{P}_{\theta'} ||_1\le 1$. 
	We have
	\bea
	|| \mathbb{P}_\theta -  \mathbb{P}_{\theta'} ||_1^2 &\le& 
	2K(\mathbb{P}_{\theta'}\mid \mathbb{P}_{\theta}) \\
	&\le& n [ tr(\Sigma(\theta')\Sigma^{-1}(\theta)) - \log det (\Sigma(\theta')\Sigma^{-1}(\theta)) - p  ].
	\eea
	Let $D := \Sigma^{-1}(\theta) - \Sigma^{-1}(\theta')$ and $j$ be the index such that $D_{jj} >0$. Note that there is only one index of which $D_{jj}$ is positive when $H(\theta,\theta')=1$.  
	Let $x = \Sigma(\theta')_{jj} D_{jj}$. We have
	\bea
	tr(\Sigma(\theta')\Sigma^{-1}(\theta)) - p  &=& tr(\Sigma(\theta') D)\\
	&=& x\\
	\log det (\Sigma(\theta')\Sigma^{-1}(\theta)) &=& \log det (\Sigma(\theta')D +I_p)\\
	&=& \log (1+x),
	\eea
	thus
	\bea
	|| \mathbb{P}_\theta -  \mathbb{P}_{\theta'} ||_1^2 &\le& n(x- \log (1+x)) \\
	&=& \frac{n}{(1+\xi)^2} \frac{x^2}{2},
	\eea
	where $\xi \in [-|x|,|x|]$.
	Since $|x| \le (\tau/\sqrt{n}) /(M - \tau/\sqrt{n})$ and $\tau/\sqrt{n} \le M/3$,
	\bea
	|| \mathbb{P}_\theta -  \mathbb{P}_{\theta'} ||_1^2&\le& 
	\frac{nx^2}{2(1+\xi)^2}  \\
	&\le& 	\frac{nx^2}{2(1-0.5)^2}  \\
	&\le&  \frac{9\tau^2}{2M^2}.
	\eea
	Since $\tau/M \le 1/3$, the proof is completed.
	
\end{proof}

\begin{proof}[Proof of Lemma \ref{main-lemma:lowerbound1}]
	By the definition of $\Sigma(\theta)^{-1}$, 
	\bea
	\frac{\Sigma(\theta)^{-1} \mathbf{1}}{\mathbf{1}^T\Sigma^{-1}(\theta) \mathbf{1}}
	=\frac{1}{p} \frac{( (\Sigma(\theta)^{-1})_{11} ,\ldots,(\Sigma(\theta)^{-1})_{pp} )^T }{M+|\theta|\tau/(p\sqrt{n})},
	\eea
	and
	\bea
	\frac{\Sigma(\theta)^{-1} \mathbf{1}}{\mathbf{1}^T\Sigma^{-1}(\theta) \mathbf{1}} - 
	\frac{\Sigma(\theta')^{-1} \mathbf{1}}{\mathbf{1}^T\Sigma^{-1}(\theta') \mathbf{1}}
	= \frac{(v_1,v_2,\ldots,v_p)^T}{ p(M + |\theta| \tau/(p\sqrt{n})) (M + |\theta'| \tau/(p\sqrt{n}))},
	\eea
	for some $v_1,v_2,\ldots,v_p \in \bbR$. 
	When $i\le p'$ and $\theta_i\neq\theta'_i$, 
	\bea
	v_i = M\tau \frac{|\theta'|-|\theta|}{p\sqrt{n}} + |\theta_i-\theta_i'| \frac{\tau}{\sqrt{n}} \Big(M+\tau \frac{|\theta'|\theta_i + |\theta|\theta'_i}{p\sqrt{n}}\Big),
	\eea
	and
	\bea
	|v_i| &\ge& \frac{\tau}{\sqrt{n}} (M+\tau \frac{|\theta'|\theta_i + |\theta|\theta'_i}{p\sqrt{n}}) - M\tau/(2\sqrt{n})\\
	&\ge& M\tau/(2\sqrt{n}),
	\eea
	since $||\theta'|-|\theta||<p/2$.
	Thus, we have
	\bea
	|| (v_1,\ldots,v_p)^T||^2_2 &\ge& H(\theta,\theta') M^2\tau^2/(4n),\\
	\min_{H(\theta,\theta')\ge 1}\frac{||w_{GMV}(\Sigma(\theta'))-w_{GMV}(\Sigma(\theta))||^2}{H(\theta,\theta')} &\ge& 
	\frac{9\tau^2}{16p^2n},
	\eea
	where the last inequality is satisfied since $|\theta|\tau/(p\sqrt{n}) \le M/3$.
	By applying this inequality and the result of Lemma \ref{lem:lower1} to Assouad's lemma, we complete the proof.

\end{proof}

Next, we give the proof of Lemma \ref{main-lemma:lowerbound2}. For this proof we present Lemma \ref{lemma:lowerbound2}.
	
	\begin{lemma}\label{lemma:lowerbound2}
    Suppose the notation of $\calG^{(2)}$ in \eqref{main-calG2} and assume $2k\epsilon_{n,p}M^{-1}<1$.
		Let 
		$
		(\Sigma(\theta') - \Sigma(\theta))\Sigma(\theta)^{-1} \mathbf{1} = (x_1,\ldots,x_p)$.
		For $i\le r$ with $\gamma_i(\theta') \neq \gamma_i(\theta)$, 
		$$x_i =  (\gamma_i(\theta') - \gamma_i(\theta)) k \epsilon_{n,p} (M^{-1}+ \delta_i),$$ 
		with
		$$|\delta_i| \le \frac{2k\epsilon_{n,p}}{M^2 (1-2k\epsilon_{n,p}/M)}.$$
		
	\end{lemma}

	\begin{proof}

		By the definition of the infinity norm and Theorem 2.3.4. in \cite{golub1996matrix}, 
		we have 
	\bea
	||(\Sigma(\theta)^{-1} - M^{-1} I_p) \mathbf{1}||_\infty
	&\le& || \Sigma(\theta)^{-1} - M^{-1} I_p||_\infty \\
	&\le& 	\frac{1}{M^2}\frac{||\Sigma(\theta) - M I_p ||_\infty }{1-M^{-1}||\Sigma(\theta)- MI_p||_\infty}.
	\eea	
	Since $M^{-1}||\Sigma(\theta)- M I_p ||_\infty \le 2k\epsilon_{n,p}M^{-1}<1$, we have 
	\bean	\label{formula:inequal_lower2}
	||(\Sigma(\theta)^{-1} - M^{-1} I_p) \mathbf{1}||_\infty \le  \frac{2k\epsilon_{n,p}}{M^2 (1-2k\epsilon_{n,p}/M)}.
	\eean
		
	Next, we check $x_i$ when $i\le r$, $\gamma_i(\theta')=1$ and $\gamma_i(\theta)=0$. 
	By the definition of $\Sigma(\theta')$, we have
	\bea
	x_i  &=&  [\Sigma(\theta')]_{i} \Sigma(\theta)^{-1} \mathbf{1} \\
	&=& M^{-1}[\Sigma(\theta')]_{i} \mathbf{1} + [\Sigma(\theta')]_{i} ( (\Sigma(\theta)^{-1}-M^{-1}I_p) \mathbf{1})\\
	&=& M^{-1} k\epsilon_{n,p} +   [\Sigma(\theta')]_{i} ( (\Sigma(\theta)^{-1}-M^{-1}I_p) \mathbf{1}).
	\eea
	By \eqref{formula:inequal_lower2}, 
	\bea
	| [\Sigma(\theta')]_{i} ( (\Sigma(\theta)^{-1}-M^{-1}I_p) \mathbf{1})| &\le&
	|| [\Sigma(\theta')]_{i}||_1 || (\Sigma(\theta)^{-1}-M^{-1}I_p) \mathbf{1}||_\infty \\
	&\le& k\epsilon_{n,p}  \frac{2k\epsilon_{n,p}}{M^2 (1-2k\epsilon_{n,p}/M)}.
	\eea
	Thus, we have
	\bea
	|x_i - M^{-1} k\epsilon_{n,p}| \le  k\epsilon_{n,p}  \frac{2k\epsilon_{n,p}}{M^2 (1-2k\epsilon_{n,p}/M)},
	\eea
	for $i\le r$ with $\gamma_i(\theta')=1$ and $\gamma_i(\theta)=0$. 
	Likewise, when $\gamma_i(\theta')=0$ and $\gamma_i(\theta)=1$, we have
		\bea
	|x_i + M^{-1} k\epsilon_{n,p}| \le  k\epsilon_{n,p}  \frac{2k\epsilon_{n,p}}{M^2 (1-2k\epsilon_{n,p}/M)}.
	\eea
	Collecting these inequalities, we complete the proof.
		
	\end{proof}

	\begin{proof}[Proof of Lemma \ref{main-lemma:lowerbound2}]
		We have
		\bea
		\frac{\Sigma(\theta)^{-1} \mathbf{1}}{\mathbf{1}^T\Sigma(\theta)^{-1} \mathbf{1}} - \frac{\Sigma(\theta')^{-1} \mathbf{1}}{\mathbf{1}^T\Sigma(\theta')^{-1} \mathbf{1}} &=&
		\frac{\Sigma(\theta')^{-1}}{\mathbf{1}^T \Sigma(\theta)^{-1}\mathbf{1}}(I_p - \frac{\mathbf{1}\mathbf{1}^T\Sigma(\theta')^{-1} }{\mathbf{1}^T \Sigma(\theta')^{-1}\mathbf{1}} ) (\Sigma(\theta') - \Sigma(\theta))\Sigma(\theta)^{-1} \mathbf{1},
		\eea	
		and 
		\bean
		&&	\Bigg|\Bigg|	\frac{\Sigma(\theta)^{-1} \mathbf{1}}{\mathbf{1}^T\Sigma(\theta)^{-1} \mathbf{1}} - \frac{\Sigma(\theta')^{-1} \mathbf{1}}{\mathbf{1}^T\Sigma(\theta')^{-1} \mathbf{1}}\Bigg|\Bigg|_2\nonumber\\
		 &=& 
			\sup_{v:||v||_2\neq 0 } \Bigg[
					\frac{v^T\Sigma(\theta')^{-1}}{||v||_2\mathbf{1}^T \Sigma(\theta)^{-1}\mathbf{1}}(I_p - \frac{\mathbf{1}\mathbf{1}^T\Sigma(\theta')^{-1} }{\mathbf{1}^T \Sigma(\theta')^{-1}\mathbf{1}} ) (\Sigma(\theta') - \Sigma(\theta))\Sigma(\theta)^{-1} \mathbf{1}
			\Bigg].\label{formula:gmvpspectral}
		\eean
	Define $v_0 \in\bbR^p$ such that $v_0^T \Sigma(\theta')^{-1}= (w_1,\ldots,w_p)$ and 
	$$
	w_i  = \begin{cases}
		\gamma_i(\theta') - \gamma_i(\theta), & \text{if } 1\le i\le r,\\
		\gamma_{i-r}(\theta) - \gamma_{i-r}(\theta'), & \text{if } r+1\le i\le 2r,\\
		0, & \text{otherwise},
	\end{cases}
	$$
	which satisfies
	\bean
	(w_1,\ldots,w_p) \mathbf{1} &=& 0.\label{formula:wsum}
	\eean
	
	Let $(\Sigma(\theta') - \Sigma(\theta))\Sigma(\theta)^{-1} \mathbf{1}  = (x_1,x_2,\ldots,x_p)$. 
	When $r<i \le 2r$, $x_i=0$ by the definition of $\Sigma(\theta)$ in $\calG^{(2)}$ and the fact that $2r\le p/2$.
	Then, using Lemma \ref{lemma:lowerbound2}, we have
	\bea
	(w_1,\ldots,w_p) 	(\Sigma(\theta') - \Sigma(\theta))\Sigma(\theta)^{-1} \mathbf{1}  &=&  \sum_{i=1}^p x_i w_i\\
	& = &H(\theta,\theta') k\epsilon_{n,p} (M^{-1} + \delta),
	\eea
	with 
	$$|\delta| \le \frac{2k\epsilon_{n,p}}{M^2 (1-2k\epsilon_{n,p}/M)}.$$
	Thus, we get
	\bean
	|(w_1,\ldots,w_p) 	(\Sigma(\theta') - \Sigma(\theta))\Sigma(\theta)^{-1} \mathbf{1}| &\ge&   
	k\epsilon_{n,p}H(\theta,\theta') (M^{-1}- \frac{2k\epsilon_{n,p}}{M^2 (1-2k\epsilon_{n,p})}	).\label{formula:wproduct}
	\eean
	We also have 
	\bean
	||v_0|| &\le& ||\Sigma(\theta)^{-1}||~ ||w|| \nonumber\\
	&\le& \sqrt{2H(\theta,\theta')}/\lambda_{\min}(\Sigma(\theta)).\label{formula:vnorm}
	\eean
	 
	 Collecting \eqref{formula:gmvpspectral}, \eqref{formula:wsum}, \eqref{formula:wproduct} and \eqref{formula:vnorm}, we obtain
	 \bea
	 			\Bigg|\Bigg|	\frac{\Sigma(\theta)^{-1} \mathbf{1}}{\mathbf{1}^T\Sigma(\theta)^{-1} \mathbf{1}} - \frac{\Sigma(\theta')^{-1} \mathbf{1}}{\mathbf{1}^T\Sigma(\theta')^{-1} \mathbf{1}}\Bigg|\Bigg|_2
	 &\ge&
	 	\frac{v_0^T\Sigma(\theta')^{-1}}{||v_0||_2\mathbf{1}^T \Sigma(\theta)^{-1}\mathbf{1}}(I_p - \frac{\mathbf{1}\mathbf{1}^T\Sigma(\theta')^{-1} }{\mathbf{1}^T \Sigma(\theta')^{-1}\mathbf{1}} ) (\Sigma(\theta') - \Sigma(\theta))\Sigma(\theta)^{-1} \mathbf{1} \\
	 &\ge& \frac{\sqrt{2H(\theta,\theta')}\lambda_{\min}(\Sigma(\theta))}{ \mathbf{1}^T \Sigma(\theta)^{-1} \mathbf{1}} k\epsilon_{n,p}
	 (M^{-1}- \frac{2k\epsilon_{n,p}}{M^2 (1-2k\epsilon_{n,p})}	)\\
	 &\ge& \frac{\sqrt{2H(\theta,\theta')}}{p} k\epsilon_{n,p}
	 (M^{-1}- \frac{2k\epsilon_{n,p}}{M^2 (1-2k\epsilon_{n,p})}	).
	 	 \eea
	 	 
	Let $\epsilon_{n,p} = (0.25(\log 2) \min(1,M_1)C_0^{-1})^{1/(1-q)} ((\log p)/n)^{1/2}$ and $k= \max(\lfloor  c_{n,p}\epsilon_{n,p}^{-q}\rfloor,0)$.
	Then, there exist a positive constant $C_2$ such that
	\bea
	\min_{H(\gamma(\theta),\gamma(\theta'))\ge 1} \frac{||w_{GMV}(\Sigma(\theta))-w_{GMV}(\Sigma(\theta'))||^2}{H(\gamma(\theta),\gamma(\theta'))} \frac{r}{2} \ge
	\frac{C_2 c_{n,p}^2  }{p} \Big(\frac{\log p}{n}\Big)^{1-q},
	\eea
	for all sufficiently large $n$. 
	By Lemma 6 in \cite{cai2012optimal}, 
	\bea
	\min_{1\le i\le r} || \bar{\mathbb{P}}_{i,0} \wedge \bar{\mathbb{P}}_{i,1} || \ge C_3,
	\eea
	for a positive constant. Combining the two inequalities, we complete the proof.

	\end{proof}

\bibliographystyle{dcu}
\bibliography{cov-ppp}